\documentclass[graybox]{svmult}

\usepackage{cite}
\usepackage{mathptmx}       
\usepackage{helvet}         
\usepackage{courier}        
\usepackage{type1cm}        
\usepackage{makeidx}         
\usepackage{graphicx}        
\usepackage{multicol}        
\usepackage[bottom]{footmisc}
\usepackage[tbtags]{amsmath}
\usepackage{amsfonts,amssymb,mathrsfs,amscd}

\usepackage{graphicx}

\usepackage[matrix,arrow,curve]{xy}

\begin{document}

\title*{Rational and $p$-adic analogs \\of J.H.C. Whitehead's conjecture}
\titlerunning{Rational and $p$-adic analogs of J.H.C. Whitehead's conjecture}
\author{A.\,M.~Mikhovich}

\institute{Andrey Mikhovich \at
              Moscow Center for Fundamental and Applied Mathematics, Moscow, Russia \\
              \email{mikhandr@mail.ru}\\           
          }


\maketitle

\abstract{
We show that subpresentations of aspherical prounipotent presentations over fields of zero characteristics and subpresentations of aspherical pro-$p$-presentations are aspherical; an application to subpresentations of aspherical discrete presentations is also included.
Following Bousfield-Kan, Quillen and Sullivan the results are regarded as affirmative answers to rational and $p$-adic analogs of J.H.C. Whitehead's conjecture.}


\section{Introduction}
\label{s0}
J.H.C. Whitehead`s conjecture (also known as Whitehead's asphericity conjecture) is a claim in algebraic topology, it states that every connected subcomplex of a two-dimensional aspherical $CW$-complex is aspherical. The question was formulated by J. H. C. Whitehead in 1941 \cite{W} and is still an open problem with many implications in topology, noncommutative geometry and group theory \cite{BH}, \cite{Bog}. \cite{Ro}.
Despite the maturity of the problematic, that arose to prove the asphericity of the knot complement in the 3-sphere, the hypothesis still emphasizes today that ``the level of our ignorance'' \cite{Bass} in this area remains extremely high after the last 40 years. If the Whitehead's asphericity conjecture is true then there is a counterexample to the Eilenberg-Ganea conjecture \cite{BB} (the group of cohomological dimension 2 that does not have 2-dimensional Eilenberg-Maclane space), promising negative solutions to a bunch of other important open problems in algebraic topology and noncommutative geometry, such as the hypotheses of Atiyah, Kadison-Kaplansky and Novikov.

Collapsing the maximal subtree of the 1-skeleton of a 2-dimensional $CW$-complex $K$ into a point does not change the homotopy type of $K$ and therefore the statement of J.H.C. Whitehead's asphericity conjecture is equivalent to the statement that each subpresentation of the aspherical presentation of a group is aspherical.
A presentation $(X|R)$ of a discrete group $G$ is by definition an exact sequence
\begin{equation}
1 \rightarrow R \rightarrow \Phi \xrightarrow{\pi} G \rightarrow 1 \label{eq1}
\end{equation}
in which $\Phi=\Phi(X)$ is the free group with a set $X$ of generators, and $R$
is a normal subgroup in $\Phi$ normally generated by a set $Y\subset R$ of defining relations. By a subpresentation of the presentation $(X|R)$ we mean the presentation $(X|R_0), R_0\subset R,$ where $R_0$ is normally generated by the subset $Y_0\subset Y$.
A group presentation $(X\mid R)$ is called \emph{aspherical} if the two-dimensional $CW$-complex $ K(X\mid R)$ associated with this presentation is aspherical, i.e. $\pi _{2}(K(X\mid R))=0$, the exact construction of $ K(X\mid R)$ will appear in Section \ref{s1.0}, we also refer the reader to the standard homotopy theory textbooks as \cite{Hut}.

A pro-$ p $-group is a group isomorphic to the inverse limit of finite $ p $-groups.
This is a topological group (with inverse limit topology)
which is compact and totally disconnected.
For such groups one has a presentation theory similar in many aspects to the combinatorial
theory of discrete groups.
The origins of cohomological and combinatorial theory of pro-$p$-groups lie in the early
papers of J.-P. Serre and J. Tate and they took the modern form in the monograph \cite{Se4}.
Several important results in discrete group theory arose as analogs of similar statements
on pro-$p$-groups. For example, the celebrated Stallings theorem,
stating that a discrete group is free if and only if its cohomological dimension
equals one, arose from the analogy, proposed by J.-P. Serre,
with the known result from pro-$p$-group theory.
Such parallelism repeats again and again especially in problems of (co)homological origin. For instance, Shafarevich's paradigm \cite[\S18,C]{Sh} about why Demushkin pro-$p$-groups have similar defining relations as discrete presentations of surface groups is solved cohomologically \cite{Ec} because as in both discrete and pro-$ p $ cases, these groups are Poincar\'{e} duality groups of dimension 2.
We also mention a recent solution of Hanna Neumann conjecture, where cohomological considerations of pro-$p$-groups help to resolve the old ``discrete'' problem \cite{ZZ}.

Serre`s philosophy began to gain ground in D. Sullivan's works \cite[Chapter 3]{Su1}, \cite[Theorem, p.53]{Su} for nilpotent spaces of finite type in the form of the so-called ``arithmetic squares''. Sullivan uses completions or localizations in order to ``fracture'' a homotopy type into ``mod-p components'' together with coherence information over the rationals.
These ideas were further developed by Bousfield and Kan for nilpotent simplicial spaces \cite[Lemma 6.2, Lemma 6.3, Lemma 8.1]{BK} and it turned out that such space $X$ (we refer the reader to \cite{May} concerning simplicial methods) can be performed as a pullback of the following diagram
$$\xymatrix{X \ar[r]\ar[d]  & \prod_{p\in \pi} (\mathbb{F}_p)_{\infty}X\ar[d] \\
\mathbb{Q}_{\infty}X\ar[r] & \mathbb{Q}_{\infty}(\prod_{p\in \pi} (\mathbb{F}_p)_{\infty}X)},$$
where $\pi$ is a set of all primes, the top map is induced by the so-called $(\mathbb{F}_p)_{\infty}$-completions and the bottom map is induced by $\mathbb{Q}_{\infty}$-completion of the top map (``arithmetic square'' also occurs for spectra \cite{Bau}).
$\mathbb{Q}_{\infty}$ and $(\mathbb{F}_p)_{\infty}$-completion functors are related to the so-called $p$-adic and $k$-prounipotent (the exact Definition \ref{d5} of the prounipotent group will appear in Section \ref{s1.2}, where we will also see that pro-$p$-groups can be regarded as prounipotent affine group schemes over $\mathbb{F}_p$) completions of discrete groups, we recall the definition
\begin{definition}\label{d4}\cite[A.2.]{HM2003},\cite[2]{Mikh2016}
Let us fix a group $G$ and a field $k$ of zero characteristics, define the prounipotent completion of $G$ as the following universal diagram, in which $\rho$ is a
Zariski dense homomorphism from $G$ to the group of $k$-points
of a prounipotent affine group scheme $G_u^{\wedge}$:
$$\xymatrix @R=0.5cm{
                &         G^{\wedge}_u(k)  \ar[dd]^{\tau}     \\
 G \ar[ur]^{\rho} \ar[dr]_{\chi}                 \\
                &         H(k)              }$$
We require that for each Zariski dense homomorphism
$\chi$ there exists a unique homomorphism $\tau$ of prounipotent groups making the diagram commutative.
When $k=\mathbb{F}_p$, we define the prounipotent completion of $G$ as the pro-$p$-completion, which is a pro-$p$-group $G^{\wedge}_p$ obeys the same universal diagram, where $H(k)$ is a finite $p$-group ($G^{\wedge}_p\cong\varprojlim_{|G/U_{\lambda}|=p^{k_{\lambda}}} G/U_{\lambda}$, see \cite[Example 2.1.6]{ZR} ).
\end{definition}

Bousfield and Kan show \cite[Proposition 4.1]{BK} that for each simplicial reduced space $X$ of finite type one can construct $R_{\infty}$-completion (we are interested in the cases when $R=\mathbb{Q}$ or $R=\mathbb{F}_p$) in three steps:
\begin{description}
\item[i]{Replace $X$ with the so-called ``Kan's loop group'' $GX$ \cite[\S 26]{May} - a simplicial group that has the homotopy type of ``loops on $X$'';}
\item[ii]{Apply the $R$-prounipotent completions functor dimension-wise, that is, we obtain a simplicial prounipotent group $(GX)^{\wedge}_R$;}
\item[iii]{Take the classifying space $\overline{W}(GX)^{\wedge}_R$ of $(GX)^{\wedge}_R$.}
\end{description}

The Kan loop-group functor $G$ and `the classifying space functor $\overline{W}$ are the adjoint functors \cite[Theorem 27.1]{May} and, in addition, they establish the so-called Quillen's equivalence of the corresponding homotopy categories \cite[Chapter 2, \S3]{Qui1}. Thus, if we understand $(GX)^{\wedge}_R$, then we understand $R_{\infty}X=\overline{W}(GX)^{\wedge}_R$ and therefore by rational and $p$-adic analogs of a discrete problem we mean the corresponding problems of prounipotent and pro-$p$ presentations (in particular, the conjecture of J.H.C. Whitehead).

Bousfield and Kan viewed the rational case as ``often take care of itself'' \cite[p.2]{BK}, in contrast, recent results from \cite{Mikh2016}, \cite{Mikh2017} show that this is not the case in two-dimensional homotopy.
For example, one-relator presentations over fields of zero characteristics have prounipotent cohomological dimention two i.e. behave as their discrete analogs. In contrast, there are uncountably many non isomorphic finitely generated pro-$p$-groups with a single defining relation \cite{Rom} (this hypothesis is popularly known as the Lubotzky conjecture) and it seems that they are much more complicated \cite[4]{Mikh2017} than their discrete one-relator cousins. Since Berrick and Hillman \cite[Corollary 4.8]{BH} proved that Whitehead's asphericity conjecture is equivalent to the statement about the rational cohomological dimension, zero characteristics can also be important for aspherical presentations.

Cohomological and presentation theories are similar for pro-$p$-groups and for prounipotent groups in zero characteristics (see Section \ref{s3.0}), moreover, there is a connection Remark \ref{r2} between zero and positive characteristics. And
as Serre`s ideology usually works with problems of (co)homological flavour, we believe that there is a prounipotent approximation theory for homology types of quasirational presentations (especially in the rational case).

A short, but rather inprecise description of the main results of this paper is the following (see Theorem \ref{t4} in Section \ref{s3.3})
\begin{theorem}
Let $(X|R_0)$ be a subpresentation of an aspherical prounipotent presentation $(X|R)$  (if the base field $k$ has a positive characteristics, we assume that $(X|R)$ is a pro-$p$-presentation), then $(X|R_0)$ is also aspherical.
\end{theorem}
D. Anick's paper \cite{An} contain a simply connected version of the rational J.H.C Whitehead's problem i.e. it is assumed that the fundamental group is trivial: ``In rational homotopy we generally consider simply connected spaces only and tensor all homotopy groups with $\mathbb{Q}$''. At the time of writing, it was known, that the rational homotopy theory generalizes, at least, to the case of non-simply-connected nilpotent spaces \cite{BG} and therefore the restriction more reflects the lack of possibility of extending the methods of \cite{An} to a non-simply connected case.

The organization of the paper is as follows. In Section \ref{s1.0} we give a concise introduction to two-dimensional complexes and their second homotopy groups to explain that J.H.C. Whitehead's original statement is equivalent to the purely algebraic problem of crossed modules arising from simplicial presentations. In Section \ref{r1.1} we show that the simplicial presentation is a combinatorial model of based loops $\Omega K(X|R)$ on the standard 2-dimensional complex $K(X|R)$.

Section \ref{s1.1} is dedicated to affine group schemes, their representations and cohomology. We have followed Jantzen's account \cite{Jan} closely, as it fits our purposes perfectly. In Section \ref{s2.3} we derive the Lyndon-Hochschild-Serre spectral sequence ``manually'', since it is in this form that it will be useful for the proof in Section \ref{s1.4}.
Section \ref{s4} contains a new key concept of a topological module, extending the definition given in the R.Hain's pioneering work \cite{Hain1992}.

In Section \ref{s3.0} we have diligently unified the presentation of the theory of prounipotent groups for characteristics zero and for pro-$p$-groups (simplicial prounipotent presentations, subpresentations and prounipotent crossed modules).
This is done in particular to make sure that the parallelism proposed by J.-P. Serre can equally arise from the theory of prounipotent groups over a field of characteristic zero; in this case, the technically sometimes simpler pro-$p$-case (sometimes the results could be obtained along the lines of \cite{CCH}) fits perfectly into the general theory developed in the language of affine group schemes.
The Section also contains the exact statement of Theorem \ref{t4}, the main result of the paper.

In Section \ref{s3} we define relation modules of prounipotent presentations and establish a criterion for asphericity in the relation module language. We prove a prounipotent analogue of Gasch\"{u}tz Lemma and also reduce the theory of non-proper presentations to proper presentations. This part contains a proof of a criterion for cohomological dimension 2 for a prounipotent group, defined by a subpresentation, adapting Tsvetkov's ideas for pro-$p$-groups.

Section \ref{s6} contains an application of Theorem \ref{t4} - Corollary \ref{c4}, which states that Bousfield-Kan completion is aspherical for finite subcomplexes of contractible 2-dimensional complexes. An interpretation of the results in the Baumslag-Dyer-Heller context is also proposed.

In contrast to our previous works \cite{Mikh2017}, \cite{Mikh2016} we use right topological modules, as they are better compatible with the existing topological literature \cite{Jan}, \cite{BH1}.

Let $p\geq2$ be a prime number. We use standard notation: $\mathbb{Z}_p$ for $p$-adic integers; $\mathbb{Q}_p$ for rational $p$-adic numbers; $\mathbb{F}_p$ for a prime field of positive characteristics, i.e. $char(k)=p>0$.

\section{Homotopy theory of discrete presentations}\label{s1.0}
Section \ref{r1.2} contains a summary of the homotopy theory of discrete group presentations, basically we follow the Brown-Huebschmann approach \cite{BH1}. In Lemma \ref{l3.3} we show that the simplicial purely algebraic definitions of a crossed module of a discrete presentation and of its second homotopy group coincide with the classical ones going back to Whitehead and Reidemeister. We've also included the Brown-Loday's description in Proposition \ref{l3}, which is convenient for the prounipotent case.

In Section \ref{r1.1} we follow the ideas of Kan, Milnor and Gabriel-Zisman to prove that the simplicial group presentation is indeed a combinatorial model of based loops on $K(X|R)$ (all precise definitions will appear below).
\setcounter{theorem}{0}
\subsection{Presentations, crossed modules and second homotopy groups}\label{r1.2}
Let we are given a presentation $(X|R)$ of the group $G\cong \Phi(X)/R$ and $K(X|R)$ is the standard 2-dimensional $CW$-complex of $(X|R)$ sometimes called the geometric realization of $(X|R)$, which is constructed as follows: it has one 0-cell $*$; 1-cell $e^1_x$ for each element $x\in X$ (hence $\pi_1(K^1,*)$ is the free group $\Phi(X)$ on $X$), and a 2-cell $e^2_y$ for each defining relation $y\in R$, attached by a representative $r_y$ of the relator $y\in R\lhd\Phi(X)$. It is known, that the homotopy type of $K(X|R)$ does not depends on the choice of representative attaching maps for 2-cells.

Let us denote $\widetilde{K}$ - the universal cover of $ K(X\mid R)$, then (since $\pi_1 (\widetilde{K})\cong 1$ and fibers are discrete) the long exact sequence of homotopy groups of the covering implies the isomorphism $\pi_2(\widetilde{K})\cong \pi_2 (K(X\mid R))$. Direct application of Hurewicz's theorem \cite[Theorem 4.37]{Hut} leads to Reidemeister's (1934) homological description of the second homotopy group $$\pi_2 (K(X\mid R))\cong H_2(C(\widetilde{K})),$$ where $C_*(\widetilde{K},\partial)$ is the chain complex of the universal cover $\widetilde{K}$ of $K(X|R)$ \cite[Corollary 2, p. 167]{BH1}.

It turns out that the chain complex $C_*(\widetilde{K})$ of the universal cover $\widetilde{K}$ is $G$-isomorphic to the chain complex $C(X|R)$ of free $G$-modules associated with the presentation $(X|R)$ of the group $G$ \cite[Proposition 9]{BH1}, which is constructed as follows:
$$C(X|R): C_2(X|R)\xrightarrow{d_2} C_1(X|R)\xrightarrow{d_1} C_0(X|R),$$
where $C_0(X|R)=\mathbb{Z}G, C_1(X|R)=\oplus_X \mathbb{Z}G, C_2(X|R)=\oplus_Y \mathbb{Z}G$ with bases as for right $\mathbb{Z}G$-modules respectively 1; $e^1_x, x\in X;$ and $e^2_y, y\in Y.$
Let $\pi:\Phi\rightarrow G, \mathbb{Z}\Phi\rightarrow\mathbb{Z}G$ both be projections determined by the presentation. Then the boundaries are given by
$d_1(e^1_x)=1-\pi x, x\in X$, $d_2(e^2_y)=\sum_X e^1_x\cdot\pi(\partial r_y/\partial x), y\in Y,$
where $\partial r_y/\partial x$ are the so called Fox derivations of $r_y$ \cite[Chapter 4]{BH1}. We need several definitions in order to obtain a purely algebraic definition of the second homotopy group

\begin{definition} \label{d1.1}
By a pre-crossed module one calls a triple $(G_2,G_1,\partial)$, where $G_1,G_2$
are groups, $\partial: G_2 \to G_1$ is a homomorphism of groups, $G_1$ acts on $G_2$
from the right, satisfying the identity
$$\text{CM 1) } \partial(g_2^{g_1}) = g_1^{-1} \partial(g_2) g_1,$$
where the action is written in the form $(g_2,g_1) \to g_2^{g_1}$, $g_2 \in G_2, g_1 \in G_1.$
\end{definition}

\begin{definition}\label{d1.2}
A pre-crossed module is called crossed if in addition for all $g_1, g_2\in G_2$ holds the Peiffer identity
$$\text{CM 2) }g_2^{\partial(g_1)}= g_1^{-1} g_2 g_1.$$
\end{definition}

\begin{definition} \label{d1.3}
A (pre-)crossed module $(G_2,G_1,d)$ is called a free (pre-)crossed module with the base
$Y\in G_2$ if $(G_2,G_1,d)$
has the following universal property with respect to maps $\nu: Y\rightarrow G'_2$ i.e.
for any (pre-)crossed module $(G'_2,G'_1,d')$ and a map $\nu: Y\rightarrow G'_2$
and for any homomorphism of groups $f:G_1\rightarrow G'_1$ such that
$fd(Y)=d'\nu(Y),$ there exists a unique homomorphism
of groups $h:G_2\rightarrow G'_2$ such that $h(Y)=\nu(Y)$
and the pair $(h,f)$ is a homomorphism of (pre-)crossed modules.
\end{definition}

J.H.C. Whitehead's (1949) description of the second homotopy group of a presentation $(X|R)$ uses the notion of a free crossed module of a presentation \eqref{eq1} in the context of relative homotopy group. In our particular case, consider a pair of spaces $(K,L)$ arising from the standard 2-dimensional $CW$-complex $K=K(X|R)$ of the presentation, as indicated above $K(X|R)=L\cup \{e^2_y\}_{y\in Y},$ where $L=K^1(X|R)$ is the 1-dimensional skeleton with the base point $*$. We can take the elements $a_y\in \pi_2(K,L)$ as homotopy classes of characteristic maps $h_y: (E^2,S^1)\rightarrow (K,L)$ of the 2-cells $e^2_y$ together with a choice of paths in $L$, one for each $y,$ joining $h_y(1)$ to $*$. Whitehead's important discovery is that $\pi_2(K,L)\xrightarrow{\partial}\pi_1(L)$
is a free crossed module on the elements $a_y\in \pi_2(K,L),$ $\pi_1(L)$ acts in the standard way \cite[Corollary 14.2]{HSJ}, \cite[4.1]{Hut}.

We give a formal algebraic description of this construction, which is equivalent to the Whitehead's \cite[Theorem 10]{BH1}) and is perfectly convenient in the prounipotent situation. Let $\Phi(X\cup Y)$  be a free group generated by the union of sets $X$ (generators) and $Y$ (elements which one to one correspond to defining relations $r_y$ in $R$ which we identify with the chosen representatives). Consider the normal closure $H$ of $Y$ in $\Phi(X\cup Y)$ and let $\theta: H\rightarrow \Phi(X)$ be a homomorphism of groups defined on basis elements $u^{-1}yu,$ where $y\in Y, r_y\in \Phi(X)$ and $u\in \Phi(X)$  by the rule

$$\theta (u^{-1}yu)= u^{-1}r_yu.$$

The pair $(H,\Phi(X),\theta)$ called a precrossed module of a presentation \eqref{eq1} since it satisfy the axiom
$$\theta (a^u)=u^{-1}\theta(a)u.$$

Now, by definition, the identities among relation are elements from the kernel of $\theta$. Among such elements there are so called Peiffer basic identities which are ``always identities'' for $a,u\in H$
$$a^{\theta(u)}= u^{-1} a u.$$
We call the subgroup $P$ of $H$ generated by all Peiffer basic identities the Peiffer group of a precrossed module. It is normal and $\Phi(X)$-invariant \cite[Proposition 2]{BH1}, therefore, taking the factor group $H/P$, we obtain the crossed module $(H/P,\Phi(X),\widetilde{\theta}),$ where $\widetilde{\theta}$ is induced from $\theta,$ for which obvious universal property holds \cite[Proposition 3, Corollary ]{BH1}.
And therefore the abelianization $\overline{H/P}$ of this free crossed module is a free $\mathbb{Z}[G]$-module on the images of the relations \cite[Proposition 7]{BH1}.
It turns out that the main benefit of the abelianization \cite[Proposition 4]{BH1}
$0\rightarrow\pi_2\rightarrow \overline{H/P}\xrightarrow{\overline{\theta}} \overline{R}\rightarrow 0$
of the free crossed module $0\rightarrow\pi_2\rightarrow H/P\xrightarrow{\widetilde{\theta}} R\rightarrow 1$ of the presentation $(X|R)$ is obtained after embedding $i:\overline{R}\rightarrow \otimes_X \mathbb{Z}G$, which is given by the rule $i(r)=\sum e^1_x \pi(\partial r_y/\partial x)$ \cite[Corollary 1]{BH1}, since it fits into commutative diagram as follows

$$\xymatrix{\overline{C} \ar[r]^{\cong}\ar[d]^{\overline{\theta}} & C_2(X|R) \ar[r]^{\cong} \ar[d]^{d_2} & C_2(\widetilde{K}) \ar[d]^{\partial_2} \\
    \overline{R} \ar[r]^{i} & C_1(X|R) \ar[r]^{\cong} \ar[d]^{d_1} & C_1(\widetilde{K}) \ar[d]^{\partial_1}\\
  & C_0(X|R) \ar[r]^{\cong} & C_0(\widetilde{K}) }.$$
This diagram shows, in particular, that $\pi_2=Ker\;\widetilde{\theta}\cong \pi_2(K(X\mid R)).$

It turns out that the free crossed module $(H/P,\Phi(X),\widetilde{\theta})$ of the presentation $(X|R)$ also arises as the second step on a way of constructing a free simplicial resolution from the given presentation. By definition, a simplicial resolution of a group $ G $ is a free simplicial group $F_{\bullet}$ (i.e. $F_i$ is a free group) such that $\pi_0(F_{\bullet})\cong G$ and $\pi_i(F_{\bullet})=0$ for $i>0$, this concept is analogous to the Eilenberg - MacLane space $ K (G, 1) $ in the category of simplicial groups. By \textbf{homotopy groups of a simplicial group} $F_{\bullet}$ we understand the homology groups of its \textbf{Moore complex} $(NF_n=\cap^{n-1}_{i=0}Ker\;d_i^n, d_n^n|_{NF_n})$.

There is a ``pas-$\grave{a}$-pas'' method for constructing a free simplicial resolution, which goes back to Andr\'{e} (see \cite[3]{MP} and its bibliography) and is similar to gluing $ n $-dimensional cells along mappings representing nonzero elements of homotopy groups in the process of constructing $K(G,1)$. We will denote the result of first two steps (taking into account that the simplicial group is degenerate in dimensions greater than one) as $F^{(1)}_{\bullet}$:
\begin{equation}\label{2}
\xymatrix{
& \Phi(s_1^1 s_0^1(X)\cup s_0^1(Y)\cup s_1^1(Y)) \ar@<1ex>[r] \ar@<-1ex>[r]_(0.6){d_0^2, d_1^2,d_2^2} \ar[r] & {\Phi(s_0^0(X)\cup Y)} \ar@<-3ex>[l]_(0.4){s_0^1,s_1^1} \ar@<-2ex>[l] \ar@<0ex>[r] \ar@<-1ex>[r]_(0.6){d_0^1,d_1^1} &
\Phi(X) \ar@<-1ex>[l]_(0.4){s_0^0} \ar[r] & G,}
\end{equation}
here $d_0^1, d_1^1, s_0^0$ for $x \in X, y \in Y, r_y \in R$ are defined by the identities
$d_0^1(x)=x,  d_0^1(y)=1, d_1^1(x)=x,  d_1^1(y)=r_y$ and we will identify the image of $s_0^0(X)$ with $X$, i.e. $s_0^0(x)=x$.

Generators $s_1^1 s_0^1(X)\cup s_0^1(Y)\cup s_1^1(Y)$ of $\Phi(s_1^1 s_0^1(X)\cup s_0^1(Y)\cup s_1^1(Y))$ of $F^{(1)}_{\bullet}$ are degenerate by construction.
We denote $R=Im\;d^1_1(Ker \;d_0^1)$, then $G\cong \Phi(X)/R$ and call a pair of free groups $\Phi(s_0^0(X)\cup Y)$ and $\Phi(X)$ connected by the homomorphisms $d_0^1, d_1^1, s_0^0$ as in \eqref{2} \textbf{a simplicial presentation of the group $G$}.
Let's build a free pre-crossed module, starting with \eqref{2}. To this end,
first we go to the Moore complex $N(F^{(1)}_{\bullet})_1\xrightarrow{d_1^1} N(F^{(1)}_{\bullet})_0$ of \eqref{2},
where $N(F^{(1)}_{\bullet})_0=\Phi(X)$, $N(F^{(1)}_{\bullet})_1=Ker\;d_0^1$.
One has $N(F^{(1)}_{\bullet})_1=Ker\; d_0\cong \Phi(Y\times \Phi(X))$  \cite[Proposition 1]{BH} and we get the so called \textbf{free pre-crossed module of
discrete presentation} \eqref{2} on the set $Y$ as a pair
$\Phi(Y\times \Phi(X))\xrightarrow{d_1} \Phi(X),$
where the right action of $\Phi(X)$ is given by $a^u=s_0^0(u)^{-1}as_0^0(u), u \in \Phi(X), a \in \Phi(Y \times \Phi(X))$.

From now on we omit the superscripts in $d^1_i$ and $s^0_0$ since they are already obvious in the 2-dimensional case.
\begin{lemma} \label{l3.3}
Let we are given a presentation \eqref{2}, then there is an isomorphism of crossed modules
$$\left(H/P, \Phi(X),
\theta\right)\cong \left(\frac{Ker\; d_0}{[Ker\; d_0, Ker\; d_1]},\Phi(X),\overline{d_1}\right),$$
arising from the coincidence of the Peiffer subgroup $P$ and $[Ker\; d_0, Ker\; d_1]$ in $Ker\; d_0$.
\end{lemma}
\begin{proof}
By construction $H=Ker\; d_0, \theta=\overline{d_1},$ so let $\langle u,a\rangle = u^{-1}au(a^{d_1(u)})^{-1}=u^{-1}au(s_0 d_1(u)^{-1} a s_0 d_1(u))^{-1}$, where $u, a\in Ker\;d_0$,  be a Peiffer commutator and denote $b=s_0 d_1(u)$, $\widetilde{a}^{-1}=u^{-1}au$.
Then $\langle u,a\rangle =\widetilde{a}^{-1}b^{-1}u\widetilde{a}u^{-1}b=[u^{-1}b,\widetilde{a}]$ and therefore Peiffer subgroup is generated by elements of the form $[u^{-1} s_0 d_1(u), a].$

Any element $x \in Ker\; d_1$ can be written as $s_1 d_1(y) \cdot y^{-1},$ where $y \in Ker\; d_0$.
Indeed, $G_1 \cong NG_1 \leftthreetimes s_0 G_0,$ so $x=y \cdot s_0(y_0),$ with $y_0 \in G_0$, $y \in NG_1$.
Since $x \in Ker\; d_1$, we have that
$1 = d_1(x) = d_1(y) \cdot d_1 s_0(y_0)$ and hence $y_0 = d_1(y^{-1})$, so $x=y \cdot s_0 d_1(y^{-1})$ and, therefore,
 $P=[Ker\; d_1, Ker\; d_0]$.
\end{proof}

By Brown-Loday Lemma \cite[5.7]{BL}, \cite{MP} in any simplicial group degenerate in dimension two, in particular \eqref{2}, the subgroups $[Ker\; d_0, Ker\; d_1]$ and $Im\;\overline{d_2}$ coincide,
where $\overline{d_2}:N(F^{(1)}_{\bullet})_2\rightarrow N(F^{(1)}_{\bullet})_1$ - the restriction of $d^2_2$ to $N(F^{(1)}_{\bullet})_2$.

Let $C=Ker\; d_0/Im\;\overline{d_2}$, it is easy to check that $(C, \Phi(X),\overline{d_1})$ is a crossed module (see \cite{Por} for details) and we recollect all the above in the following
\begin{proposition} \label{l3}
Suppose we are given a presentation \eqref{2}, then there is an isomorphism of crossed modules
$$\left(C, \Phi(X),\overline{d_1}\right)\cong\left(\frac{Ker\; d_0}{[Ker\; d_0, Ker\; d_1]},\Phi(X),\overline{d_1}\right),$$
arising from the coincidence of the subgroups $[Ker\; d_0, Ker\; d_1]$ and $Im\;\overline{d_2}$ in $Ker\; d_0$ and hence the isomorphism
$$\pi_1(F^{(1)}_{\bullet})\cong \frac{Ker\; d_0\cap Ker\;d_1}{[Ker\;d_0,Ker\; d_1]}\cong Ker\;\overline{d_1},$$
which allow us to identify $\pi_2(K(X\mid R))$ with $\pi_1(F^{(1)}_{\bullet})$.
\end{proposition}




\subsection{Simplicial presentation as a model of $\Omega K(X|R)$}\label{r1.1}
The simplicial presentation \eqref{2} coincide with the free simplicial group $B_{K(X|R)}$ constructed by D.Kan \cite[Definition 5.2]{Kan}. $B_{K(X|R)}$ is loop homotopy equivalent by \cite[Theorem 5.5]{Kan} to $GS_1K(X|R)$ - the ``Kan loop group'' of
the first Eilenberg subcomplex $S_1K(X|R)$ of the total singular complex of $K(X|R)$ \cite[Definition 8.3]{May}.
Recall that the ``Kan loop group functor'' $G$ assigns to each reduced simplicial set $X$ a free simplicial group $GX$ subject to the ``principal twisted cartesian product'' \cite[Theorem 26.6]{May}
$GX\rightarrow GX \times_{\tau} X \xrightarrow{p} X,$
where $|GX \times_{\tau} X|$ is a contractible space, $p$ is a Kan fibration and therefore $GX$ is a simplicial ``loop space'' for $X$.

By \cite{Qui4} the geometric realization $|p|$ of the Kan fibration $p$ is a Serre fibration. Then by \cite[3.3, Chapter 3]{GZ} $|GX|$ is the fiber of $|p|$
$$|GX|\rightarrow |GX \times_{\tau} X| \xrightarrow{|p|} |X|$$
and, therefore, (as in \cite[Proposition 4.66]{Hut}) there is a weak homotopy equivalence $\beta:|GS_1K(X|R)|\simeq \Omega K(X|R)$.
Since $K(X|R)$ is a $CW$-complex, by \cite[Corollary 2]{Milnor} $\Omega K(X|R)$ is homotopy equivalent to a $CW$-complex.
And as geometric realization is a CW-complex \cite[Theorem 14.1]{May}, by Whitehead's theorem, $\beta$ is a homotopy equivalence. Finally, we have a homotopy equivalence $|B_{K(X|R)}|\simeq \Omega K(X|R)$ and therefore $B_{K(X|R)}$ (i.e. a simplicial presentation in the 2-dimensional case) is the combinatorial model of based loops on $K(X|R)$.

\section{Representations and cohomology of affine group schemes}\label{s1.1}
The most convenient presentation of affine group schemes for our purposes can be found in books \cite{Wat}, \cite{Jan}, the necessary material on the general theory of Hopf algebras in the required volume is available in \cite[Chapter 1]{Mont}, \cite{Car}. We have specially included some useful facts about regular representations and induced modules as they arise in the study of prounipotent presentations. A modern exposition of cohomology theory is contained in \cite{Jan}, but at the same time, to construct the spectral sequence, we follow classical approach \cite{HS} to obtain an explicit double complex.

The category of topological modules is defined in Section \ref{s4} by continuous duality. This generality automatically leads to the fact that the introduced category is abelian, since the axioms of an abelian category are self-dual \cite[1.4]{Groth}.

\subsection{Affine group schemes and their representations}\label{s1.01}
By an affine group scheme over a field $k$ one calls a representable functor $G$
from the category $Alg_k$ of commutative $k$-algebras with unit to the category of groups.
If $G$ is representable by the Hopf algebra $\mathcal{O}(G)$, then as a functor $G$ is given, for any commutative
$k$-algebra $A$, by the formula
$$G(A) = Hom_{Alg_k}(\mathcal{O}(G),A).$$
We assume that homomorphisms $Hom_{Alg_k}$ send the unit of the algebra $\mathcal{O}(G)$
to the unit of a $k$-algebra $A$. The Hopf algebra $\mathcal{O}(G)$, representing the functor $G$,
is usually called the \emph{algebra of regular functions} of $G$.
We remind that the composition law in $G(A)$ is given for $g_1, g_2\in G(A)$, $x\in \mathcal{O}(G)$ by the formula $(g_1\cdot g_2)(x)=m_A(g_1\otimes g_2)\Delta(x)$, where $\Delta:\mathcal{O}(G)\rightarrow\mathcal{O}(G)\otimes \mathcal{O}(G)$ is the coalgebra map, $m_A$ is the multiplication in $A$, the inverse of $g\in G(A)$ is a $k$-algebra homomorphism, given as the composition $g\circ s$, where $s:\mathcal{O}(G)\rightarrow \mathcal{O}(G)$ is the antipod.

The Yoneda lemma implies the anti-equivalence of the categories of affine group schemes and commutative
Hopf algebras \cite[1.3]{Wat}.
Let us say that an affine group scheme $G$ is \emph{algebraic} if its Hopf algebra of regular functions $\mathcal{O}(G)$
is finitely generated as a commutative $k$-algebra.

We call the Zariski closure of the subset $S\subseteq G(k)$ the smallest affine subgroup $H$ in $G$ such that $S\subseteq H(k)$, it is $\varprojlim H_{\alpha}$, where $H_{\alpha}$ is the closure of the image of $S$ in $G_{\alpha}(k)$.

Let $M$ be a $k$-module then it defines a $k$-group functor $GL(M)$ by the formulae $GL(M)(A)=End_A(M\otimes A)^{\times}$ (invertible endomorphisms) called a general linear group of $M$.
Let $G$ be an affine $k$-group scheme, by a representation of $G$ on $M$ we understand a homomorphism of group functors $G\rightarrow GL(M).$
This is equivalent to say that each $G(A)$ acts from the left on $M(A)=M\otimes A$ through $A$-linear maps, so representation provides for each $A$ a group homomorphism $G(A)\rightarrow End_A(M\otimes A)^{\times}.$



For any $k$-algebra $A$ and $g\in G(A)$ we have a commutative diagram
$$\xymatrix{G(\mathcal{O}(G))\times (M\otimes \mathcal{O}(G)) \ar[rrr]^(0.6){f(\mathcal{O}(G))=\Delta_M\otimes id_{\mathcal{O}(G)}}\ar[d]_{G(g)\times(id_M\otimes g)} & & & M\otimes \mathcal{O}(G) \ar[d]^{id_M\otimes g} \\
G(A)\times (M\otimes A) \ar[rrr]^{f(A)} & & & M\otimes A}$$
Lets take a look at $id_{\mathcal{O}(G)}\times (m\otimes 1) \in G(\mathcal{O}(G))\times (M\otimes \mathcal{O}(G))$, commutativity of the diagram gives the following identity:
$$f(A)\circ G(g)(id_{\mathcal{O}(G)})(m\otimes 1)=(id_M\otimes g)\circ \Delta_M (m\otimes 1),$$
where $\Delta_M: M\otimes k\rightarrow M\otimes \mathcal{O}(G)$ (a comodule map \cite[3.2]{Wat}) is a restriction to $M$ of $f(id_{\mathcal{O}(G)}):M\otimes \mathcal{O}(G)\rightarrow M\otimes \mathcal{O}(G)$.

As $g=G(g)(id_{\mathcal{O}(G)}),$ for all $m\in M$ we have
\begin{equation}\label{eq4}
g(m\otimes 1)=(id_M\otimes g)\circ \Delta_M (m\otimes 1).
\end{equation}

\begin{example}[Left and right regular representations]
Let $\mathbb{A}^1_k$ be the affine line, as $\mathcal{O}(\mathbb{A}^1_k)\cong k[X]$ is the free polynomial algebra, by Yoneda's Lemma, there is an isomorhism of $k$-vector spaces $Mor(G,\mathbb{A}^1_k)\cong\mathcal{O}(G)$.

$\Delta_l$ - left and $\Delta_r$ - right regular actions of $G$ on $M=\mathcal{O}(G)$ are given for $\phi_A\in Mor(G(A),\mathbb{A}^1_k(A))$ and
$y_A ,x_A\in G(A)$ by the formulas respectively
$(y_A\cdot \phi_A)(x_A)=\phi_A(y_A^{-1} x_A)$ and $(y_A\cdot\phi_A)(x_A)=\phi_A(x_A y_A )$.
It turns out, that $\Delta_r$ comes from comodule structure given by the comultiplication in $\mathcal{O}(G)$, i.e. $\Delta_r=\Delta_G$. 

Indeed, $(y_A\cdot \phi_A)(x_A):=m_A (y_A\cdot(\phi_A\otimes id_A)(x_A\otimes id_A))$ equals by \eqref{eq4} to
$m_A((id_A\otimes y_A)\Delta_G (\phi_A)(x_A\otimes id_A))=m_A(\Delta_G (\phi_A)(x_A\otimes y_A))$, the last equality is by linearity of $\Delta_G$ at $y_A$, and finally $m_A(\Delta_G (\phi_A)(x_A\otimes y_A))=(x_A y_A)(\phi)$ as evaluation of the product in $G(A)$. But $(x_A y_A)(\phi)=\phi_A(x_A y_A)$ as $\phi$ is the natural transformation i.e.
$\phi_A\circ G(z_A)(id_{\mathcal{O}(G)})=\phi_A(z_A)$ equals to $z_A(\phi(id_{\mathcal{O}(G)}))=z_A(\phi)$ for $z_A\in G(A)$.
A similar statement is proven in \cite[Lemma 1.6.4, Ex. 1.6.5]{Mont}.

$\mathcal{O}(G)$-comodule structure which corresponds to the left regular representation is given by $\Delta_l=t\circ(s\otimes id_{\mathcal{O}(G)})\circ \Delta_G$, where $s$ is antipod and $t(a\otimes b)=b\otimes a$.
\end{example}
There is a natural notion of a $G$-module homomorphism and as it is perfectly explained in \cite[3.2]{Wat} the corresponding category $Rep(G)$ of left $G$-representations is equivalent to the category of right $\mathcal{O}(G)$-comodules.

Let $H$ is a closed subgroup of $G$, then $\mathcal{O}(H)=\mathcal{O}(G)/I_H,$
where $I_H$ is the Hopf ideal \cite{Wat} defining the subgroup $H$, and let $M$ be a $G$-module, whence we obtain the $k$-linear map
$\mu: M\rightarrow M\otimes\mathcal{O}(G)\rightarrow M\otimes\mathcal{O}(H)$, which defines left $H$-module structure
on $M$ (that is the \textbf{restriction functor} $M\downarrow^G_H$).

To construct injective hulls, we need the concept of an induced module.
So, let $H$ be a closed subgroup of the affine group scheme $G$. For each $H$-module $M$ the \textbf{induced module} $M\uparrow^G_H$ is defined by \cite[3.3]{Jan} as follows

$$M\uparrow^G_H=\{ f\in Mor(G,M_a)\mid f(gh)=h^{-1}f(g)\}$$
for all $g\in G(A)$ and $h\in H(A), A\in Alg_k\}, $
with left regular action of $G$.

\begin{proposition}\cite[I,3.4]{Jan} \label{p4}
Let $H$ be a closed subgroup of an affine group scheme $G$ and $M$ be an $H$-module.

a) $\varepsilon_M:M\uparrow^G_H\rightarrow M$ is a homomorphism of $H$-modules

b) For each $G-$module $N$ the map $\varphi\mapsto \varepsilon_M\circ \varphi $ defines an isomorphism
$$Hom_G(N,M\uparrow^G_H)\cong Hom_H(N\downarrow^G_H,M).$$
\end{proposition}

\begin{proposition}\label{p5} ~\cite[I,~3.6]{Jan} Let $H$ be a closed subgroup of an affine group scheme
$G$ and $M$ be an $H$-module. If $N$ is a $G$--module, then there is a canonical isomorphism of $G$-modules
$$(M\otimes N\downarrow^G_H)\uparrow^G_H\cong M\uparrow^G_H\otimes N.$$
\end{proposition}

Following ~\cite[I,~3.7]{Jan}, let's discuss some useful implications from the propositions above.
Suppose that $H=1$, then for all $k$-modules $M$
$M\uparrow^G_1=M\otimes \mathcal{O}(G), $ where $M$ is considered as a trivial $G$-module
and in particular
$k\uparrow^G_1=\mathcal{O}(G).$

Combining the latter identity with Proposition \ref{p4} (b) we obtain for each $G$-module $M$ an isomorphism of $k$-linear spaces (it is actually an isomorphism of topological modules, see Section \ref{s4})
\begin{equation}\label{eq10}
Hom_G(M,\mathcal{O}(G))\cong M^*.
\end{equation}
If we put $M=k_a$ in the tensor identity, then for each $G$-module $N$ there is a remarkable isomorphism
\begin{equation}\label{e2}
N\otimes \mathcal{O}(G)\cong N\uparrow^G_1=N_{tr}\otimes \mathcal{O}(G),
\end{equation}
given by the formula
$x\otimes f\mapsto (1\otimes f)\cdot(id_N\otimes s)\circ\Delta_N(x),$ where $N_{tr}$
denotes the $k$-module $N$ with the trivial action of $G$ and $s$ is the antipode in $\mathcal{O}(G)$.

We shall need $M^G$ - the submodule of \textbf{fixed points} of a $G$-module $M$
$$M^G=\{m\in M\mid g(m\otimes1)=m\otimes1, \mbox{for all } g\in G(A),\mbox{and all }A\in Alg_k\}.$$
Let $g=id_{\mathcal{O}(G)}\in G(\mathcal{O}(G))$ in \eqref{eq4}, then $M^G=\{m\in M\mid \Delta_M(m)=m\otimes1\}$ and therefore
\begin{equation}
 M^G=Ker\;(\Delta_M-id_M\otimes 1). \label{eq2}
\end{equation}

\subsection{Cohomology of affine group schemes}\label{s1.02}

It follows from Proposition \ref{p4} that the induction functor $M\uparrow^G_H$ is a right adjoint to the restriction functor $M\downarrow^G_H$ and, therefore, $M_{tr}\uparrow^G_1\cong M_{tr}\otimes \mathcal{O}(G)$ is an injective $G$-module, since $M_{tr}$ is obviously injective as a module over the trivial group.
It follows from the unit diagram $id_M=(id_M\otimes\varepsilon_M)\circ\Delta_M$ \cite[3.2]{Wat} that the map
$M \xrightarrow{\Delta_M} M\otimes\mathcal{O}(G)$ is an inclusion and, therefore, composing $\Delta_M$ with the isomorphism \eqref{e2} we embed $M\hookrightarrow M_{tr}\otimes \mathcal{O}(G)$ into the injective module.

The category of $G$-modules is abelian ~\cite[I,~2.9]{Jan} and as we just saw
contains enough injective objects, so we can define the cohomology groups $H^n(G, M)$ of the affine group scheme
$G$ with coefficients from $G$-module $M$ as the $n$-th derived functors of the fixed points functor $M^G$.

$H^*(G,M)$ can be computed explicitly using the \textbf{Hochschild complex} which is the cohomology computation by means of standard resolutions \cite[6.2]{ZR}.
\begin{equation}\label{eq7}{
C^*(G,M)=M\otimes^n \mathcal{O}(G), \partial^n:C^n(G,M)\rightarrow C^{n+1}(G,M),}
\end{equation}
where $\partial^n=\sum^{n+1}_{i=0}(-1)^i\partial_i^n$ for  $n\in \mathbb{N}$ is defined by formulas
$$\partial_0^n(m\otimes f_1\otimes\ldots\otimes f_n)=\Delta_M(m)\otimes f_1\otimes\ldots\otimes f_n,$$
$$\partial_i^n(m\otimes f_1\otimes\ldots\otimes f_n)=m\otimes f_1\otimes\ldots\otimes f_{i-1}\otimes\Delta_G(f_i) \otimes f_{i+1}\otimes \ldots\otimes f_n,\mbox{ for $1\leq i \leq n$}, $$
$$\partial_{n+1}^n(m\otimes f_1\otimes\ldots\otimes f_n)=m\otimes f_1\otimes\ldots\otimes f_n\otimes 1.$$

We can identify $C^n(G,M)$ with the ``inhomogeneous bar complex'' \cite[I.2]{HS} $Mor(G^n,M_a),$ where $G^n$ is the direct product of $n$ copies of $G$.
Indeed, by Yoneda's lemma, $(M\otimes \mathcal{O}(G)^{n})\otimes A\cong Mor (G_A^n, (M\otimes A)_a)$ and, therefore, $\partial_i^n$ is expressed as

$$\partial_0^nf(g_1,g_2,\ldots,g_{n+1})=g_1f(g_2,\ldots,g_{n+1})$$
 $$\partial_i^nf(g_1,g_2,\ldots,g_{n+1})=f(g_1,g_2,\ldots,g_{i-1},g_{i}g_{i+1},\ldots,g_{n+1})\mbox{ for $1\leq i \leq n$},$$
 $$\partial_{n+1}^n=f(g_1,g_2,\ldots,g_{n}).$$

The Hochschild complex
arises from an exact sequence of $G$-modules
$$0\rightarrow k_a\rightarrow \mathcal{O}(G)\rightarrow \mathcal{O}(G)^{\otimes 2}\rightarrow\ldots,$$
which is just the injective resolution $0\rightarrow k_a\rightarrow C^n(G,\mathcal{O}(G))$ of $k_a$ treating $M=\mathcal{O}(G)$ as $G$-module via $\rho_r$ and $k_a$ as trivial $G$-module
\cite[4.15, (1)]{Jan}.

This sequence is in fact a sequence of homomorphisms of $G$-modules if we assume that $G$ acts on $\mathcal{O}(G)^{\otimes n}$ via $\rho_l$ on the first factor and trivially on all the other factors, it is for this action $\mathcal{O}(G)\rightarrow \mathcal{O}(G)\otimes \mathcal{O}(G), f\mapsto \Delta_G(f)- f\otimes1$ is $G$-equivariant map.
We tensor $0\rightarrow k_a\rightarrow C^n(G,\mathcal{O}(G))$ with $M$ on the left and use \eqref{e2} in order to obtain an injective resolution of $M$
\begin{equation}\label{eq5}
{0\rightarrow M\rightarrow M_{tr}\otimes \mathcal{O}(G)\rightarrow M_{tr}\otimes \mathcal{O}(G)^{\otimes 2}\rightarrow\ldots}
\end{equation}
and therefore $H(G,M)$ is the cohomology of the complex
$$0\rightarrow(M_{tr}\otimes \mathcal{O}(G))^G\rightarrow (M_{tr}\otimes \mathcal{O}(G)^{\otimes 2})^G\rightarrow\ldots$$
But as $G$ acts trivially on all but one factor, and as $\mathcal{O}(G)^G=k_a$, we see that

$(M_{tr}\otimes \mathcal{O}(G)^{\otimes n+1})^G\cong M_{tr}\otimes \mathcal{O}(G)^{\otimes n}\cong C^n(G,M)$.
\begin{remark}
Let $G\cong\varprojlim G_{\gamma}$ be a pro-$p$-group (where $G_{\gamma}$ are finite $p$-groups), by Example \ref{e6} below, $G$ may be thought as an affine group scheme with the algebra of regular functions
$\mathcal{O}(G)= \mathbb{F}_pG^{\vee}$.
Now, the right $\mathcal{O}(G)$-comodule $M$ over the affine group scheme represented by $\mathcal{O}(G)$ is a left discrete $G$-module \cite[2.1]{Se4} of exponent equal to $p$. By Lemma \ref{l2} (Section \ref{s3.4}) $M\cong\varinjlim M_{\gamma}$, where
$ M_{\gamma}$ are $G_{\gamma}$-submodules and $G_{\gamma}$ are taken from the decomposition of $G$.
Conversely, if $M$ is a left discrete $G$-module of exponent equal to $p$, then it is a $G$-module in the schematic sense.

As $(M\otimes \mathcal{O}(G)^{\otimes^n})^G\cong \varinjlim (M_{\gamma}\otimes \mathcal{O}(G_{\gamma})^{\otimes^n})^{G_{\gamma}}$ and, since
homology commutes with direct limits, we have
$H^n(G,M)\cong \varinjlim H^n(G_{\gamma}, M_{\gamma}) $, and the schematic cohomology of pro-$p$-groups with $G$-module coefficients coincide with the cohomology groups of a pro-$p$-group with coefficients in discrete $G$-modules of exponent equal to $p$ \cite[6.2]{ZR}, \cite[2.2]{Se4}.
\end{remark}

\subsection{The Lyndon-Hochschild-Serre spectral sequence}\label{s2.3}
\begin{definition}\label{d3.1}
A double complex $E$ is a collection $E^{p,q}_0, p,q\geq 0$ of abelian groups and maps $d_0^{p,q}:E^{p,q}_0\rightarrow E^{p,q+1}_0, d_1^{p,q}:E^{p,q}_0\rightarrow E^{p+1,q}_0$ arranged in the diagram such that the following conditions are satisfied:
\begin{description}
\item[(i)]{each row satisfied $d_1\circ d_1=0;$}
\item[(ii)]{each column satisfied $d_0\circ d_0=0;$}
\item[(iii)]{$d_0\circ d_1+d_1\circ d_0=0.$}
\end{description}
\end{definition}
The total complex of a double complex $X^n=Tot(E)=\bigoplus_{i+j=n}E^{i,j}_0$ is given with the differential $d=d_0+d_1:X^n\rightarrow X^{n+1}$.

The double complex inherits the filtration $D_0^{p,q}=F^pX^{p+q}=\bigoplus_{i+j=p+q}^{i\geq p} E^{i,j}_0$,
each quotient
$F^pX^{p+q}/F^{p+1}X^{p+q}\cong E^{p,q}_0$ of the filtration
should be treated as a single column of $E$ with the differential $d_0:E_0^{p,q}\rightarrow E_0^{p,q+1},$ since $d_1$ maps into a lower layer. So we get the derived couple \cite{HSJ}
$$E_1^{p,q}=H(E^{p,q}_0,d_0),\mbox{   }D_1^{p,q}=H(E_0^{p,q}\oplus E_0^{p+1,q-1}\oplus..., d_0+d_1).$$

Let $x\in E_0^{p,q}$ and $d_0(x)=0,$ so $x$ represents a class $[x]\in E_1^{p,q},$ then $k_1$ in the derived couple is defined as a boundary homomorphism associated to the short exact sequence of chain complexes
$0\rightarrow F^{p+1}X\rightarrow F^pX\rightarrow F^pX/F^{p+1}X\rightarrow 0.$
Since $d_0(x)=0$ we can calculate
$k_1[x]=[d_1(x)]\in D_1^{p+1,q}$
i.e. the differential of the derived pair $j_1\circ k_1[x]=[d_1(x)]$ coincides with the horizontal differential $d_1$ of $E$ and we obtain the following

\begin{proposition}\label{t3.1}
For a given double complex $(E_0^{p,q}, d_0, d_1)$ there is a spectral sequence with
$$E_1^{p,q}=H^q(E_0^{p,q},d_0),\mbox{     }
E_2^{p,q}=H^p(H^q(E_0,d_0),d_1)$$
$$E_{\infty}^{p,q}=F^pH^{p+q}(Tot(E_0))/F^{p+1}H^{p+q}(Tot(E_0))$$
\end{proposition}

\begin{remark}\label{r3.1}
It turns out that in a double complex $E_0^{p,q}$ the differentials $d_2^{p,q}$ may be described using the differentials $d_0$ and $d_1$ as follows. For each class $e\in E_2^{p,q},$ there are elements $x\in E_0^{p,q}$ and $y\in E_0^{p+1,q-1}$ such that:
\begin{description}
\item[(i)]{$d_0(x)=0$, $d_1(x)=-d_0(y)$}
\item[(ii)]{$d_2^{p,q}(e)=[d_1(y)]$, where by $[d_1(y)]$ we denote the class of $d_1(y)$ in $E_2^{p+2,q-1}$}
\end{description}
\end{remark}

\begin{proposition}\label{t3.2}
Let $G$ be an affine group scheme, $H\lhd G$ be a closed normal subgroup of $G$ and $M$ is $G$-module. Then there is a cohomological Lyndon-Hochschild-Serre spectral sequence
$E_2^{p,q}=H^p(G/H,H^q(H,M))\Rightarrow H^{p+q}(G,M)$.
\end{proposition}
\begin{proof}
Using the Hochschild resolution \eqref{eq5} we construct a double complex
$K^{p,q}=(M\otimes \mathcal{O}(G)^{\otimes(q+1)})^H\otimes \mathcal{O}(G/H)^{\otimes p}$
with differentials $d_0^{p,q}, d_1^{p,q}$ given by the formulas for
$a\in (M\otimes \mathcal{O}(G)^{\otimes(q+1)})^H$ and $b\in\mathcal{O}(G/H)^{\otimes p}$ as follows
$$d_0^{p,q}(a\otimes b)=(-1)^p \widetilde{d}_0^{p,q}(a)\otimes b, \mbox{    } d_1^{p,q}(a\otimes b)=a\otimes \widetilde{d}_1^{p,q}(b),$$
where $$\widetilde{d}_0^{p,q}: (M\otimes \mathcal{O}(G)^{\otimes(q+1)})^H\rightarrow (M\otimes \mathcal{O}(G)^{\otimes(q+2)})^H$$ is the differential extracted from the Hochschild complex \eqref{eq5} and
$$\widetilde{d}_1^{p,q}:(M\otimes \mathcal{O}(G)^{\otimes(q+1)})^H\otimes \mathcal{O}(G/H)^{\otimes p}\rightarrow (M\otimes \mathcal{O}(G)^{\otimes(q+1)})^H\otimes \mathcal{O}(G/H)^{\otimes (p+1)}$$ is the differential in the Hochschild complex $C(G/H,(M\otimes \mathcal{O}(G)^{\otimes(q+1)})^H)$ \eqref{eq7}.
Then $(E_0^{p,q},d_0,d_1)$ is a double complex and we define the Lyndon-Hochschild-Serre spectral sequence as the spectral sequence of Proposition \ref{t3.1}.
Indeed, since $\otimes\mathcal{O}(G/H)^{\otimes p}$ is an exact functor we have $E_1^{p,q}\cong H^q(H,M)\otimes \mathcal{O}(G/H)^{\otimes p}$ and therefore $E_2^{p,q}=H^p(G/H,H^q(H,M))$.

Now we consider the row filtration of $(E^{q,p},d_1,d_0)$. As for any $G$-module $A$ by \eqref{e2} $(A\otimes \mathcal{O}(G))^H\cong A_{tr}\otimes \mathcal{O}(G)^H$ and $A_{tr}\otimes \mathcal{O}(G)^H\cong A_{tr}\otimes \mathcal{O}(G/H)$ it follows $(M\otimes \mathcal{O}(G)^{\otimes(q+1)})^H$ are injective $G/H$-modules, hence for $q>0$
$E_2^{q,p}=H^p(G/H,(M\otimes \mathcal{O}(G)^{\otimes (q+1)})^H)=0$. But as $A^G=(A^H)^{G/H}$ for any $G$-module A, it follows $E_2^{0,p}\cong H^p(G,M)$ and therefore $H^p(Tot(E))\cong H^p(G,M)$.
\end{proof}


\subsection{Topological modules as a dual category}\label{s4}
\begin{definition}
A linearly-compact vector space over a field $k$ (we always assume that $k$ is taken with the discrete topology) is a topological vector space $V$ over $k$ such that:
\begin{description}
\item[1.]{The topology is linear: the open affine subspaces form a basis for the topology;}
\item[2.]{Any family of closed affine subspaces with the finite intersection property has nonempty intersection;}
\item[3.]{The topology is Hausdorf.}
\end{description}
\end{definition}
Linearly-compact vector spaces were introduced in order to preserve the duality that holds for finite-dimensional vector spaces. This can be achieved by introducing a topology on a dual space $V^*$, taking
$$U^{\perp}=\{\phi \in A^{\ast} \mid \phi(U)=0, \mbox{where $U$ is a finite dimensional $k$-subspace of $V$} \}$$ as the basis of the system of neighborhoods of zero in $V^{\ast}$. It turns out that $V^*\cong \varprojlim_{dim_k(U)<\infty } V^*/U^{\perp}$ and $V^*$ is the complete topological vector space.

For any linearly-compact vector space $V$ it makes sense to speak about the dual space $V^\vee$
of continuous linear functions, and the evaluation map defines the continuous duality \cite[1.2]{Die}
(here $V$ is discrete and $W$ is linearly-compact)
$$e: V\rightarrow V^{*\vee}, \quad v\mapsto(\phi \rightarrow \phi(v)),\mbox{  }
\widetilde{e}: W\rightarrow W^{\vee *}, \quad v\mapsto(\phi \rightarrow \phi(v)).$$







Recall that completed tensor product $E \widehat{\otimes}_k F$ of topological $k$-vector spaces $E$ and $F$
is the completion of $E \otimes_k F$ with respect to the topology (called the topology of tensor product)
given by the fundamental system of neighborhoods of 0 consisting of the sets
$V \otimes_k F + E\otimes_k W$, where $V$ (respectively $W$) is an arbitrary element of
the fundamental system of neighborhoods of 0 consisting of vector subspaces of $E$ (respectively $F$) \cite[1.2.4]{Die}. It follows directly from the definition of completed tensor product that $(E \widehat{\otimes}_k F)^{\vee}\cong E^{\vee} \otimes_k F^{\vee}$.

Let $A(m,\Delta,s,e,\varepsilon)$ be a commutative Hopf algebra with multiplication $m:A\otimes A\rightarrow A,$ comultiplication $\Delta: A\rightarrow A\otimes A,$ antipod $s: A\rightarrow A$, unit $e: k\rightarrow A$ and counit $\varepsilon: A\rightarrow k$, this $k$-linear maps are included in standard diagrams of the Hopf algebra definition \cite[1.1, 1.5]{Mont}.

Let $A^*=Hom_k(A,k)$ be the dual and let $m_*=\Delta^*: A^*\widehat{\otimes} A^*\rightarrow A^*$ - (the dual) multiplication, $\Delta_*=m^*: A^*\rightarrow A^*\widehat{\otimes} A^*$ - comultiplication , $\varepsilon_*=e^*$ and $e_*=\varepsilon^*$, $s_*=s^*$. Those $k$-linear maps are also included in standard diagrams of Hopf algebra definition and we come to the
\begin{definition}\label{d8} Let us say that a linearly-compact $k$-vector space $A^*$ with continuous maps $(m_*,\Delta_*,s_*,e_*,\varepsilon_*)$ of linearly-compact vector spaces is
complete Hopf algebra ($CHA$ for short) if $A^*(m_*,\Delta_*,s_*,e_*,\varepsilon_*)$ is included into the standard diagrams from the definition
of the cocommutative Hopf algebra, where the usual tensor products are replaced by the completed ones.
 \end{definition}
\begin{remark}
Our definition is equivalent to the definition of $CHA$ in Quillen's paper \cite{Qui5} if and only if $CHA$ in the Quillen's sence is finitely generated. We additionally require linearly-compactness, bearing in mind that $ CHA $ arise as dual of a Hopf algebra.
\end{remark}

Recollecting what was said above we obtain the following
\begin{corollary}\cite[\S2]{Die}\label{c2}
Assigning $A\mapsto A^*$ yields a one-to-one correspondence between the structures of
commutative Hopf algebras on a vector $k$-space $V$ with discrete topology and the
structures of cocommutative linearly-compact complete Hopf algebras on $A^*$.
\end{corollary}

\begin{example}[Quillen's formulae]\label{e7}
Let $G$ is a finitely generated discrete group, $k[G]$ a group ring with coefficients in the field $k$ of zero characteristics, $I$ - the augmentation ideal. Then $dim_k I/I^2<\infty$, hence the factors $k[G]/I^n$ are of finite dimension and, therefore, $\widehat{k}G=\varprojlim k[G]/I^n$ is a linearly-compact vector space. For $g\in G$ a diagonal coproduct $\Delta(g)=g\otimes g$ is given, by linearity we obtain a continuous in the $ I $-adic topology map $\Delta:kG\rightarrow kG\otimes kG$ and, therefore, $\widehat{\triangle}: \widehat{k}G\rightarrow \widehat{k}G\widehat{\otimes }\widehat{k}G$ is the continuous coproduct, which, together with completed multiplication and antipod, induced by the map $s(g)=g^{-1}$ for $g\in G$, define $CHA$-structure on $\widehat{k}G$.
It turns out \cite[A.3]{Qui5}, \cite[2]{Knu} that the prounipotent group $G^{\wedge}_u$ with the algebra of regular functions
$\mathcal{O}(G^{\wedge}_u)\cong(\widehat{k}G)^\vee$ is the prounipotent completion of $G$ in the sense of Definition \ref{d4}.
\end{example}

The structure of a right $ \mathcal{O}(G) $-comodule on $ M $ is determined by the comodule map
$ \Delta_M: M \rightarrow M \otimes \mathcal{O}(G) $
and hence, by duality, we obtain a continuous mapping
$ \Delta_M^*: M^* \widehat{\otimes} \mathcal{O}(G)^* \rightarrow M^* $
included in the associativity and unit diagrams
\begin{equation}\label{eq8}
$$\xymatrix{M^*\widehat{\otimes} \mathcal{O}(G)^*\widehat{\otimes} \mathcal{O}(G)^* \ar[r]^(0.6){\Delta_{M}^*\widehat{\otimes} id}\ar[d]^{id_{M^*}\widehat{\otimes} \Delta_G^*}  & M^*\widehat{\otimes} \mathcal{O}(G)^*\ar[d]^{\Delta_{M}^*} & M^*\widehat{\otimes} \mathcal{O}(G)^* \ar[r]^(0.6){\Delta_{M}^*}\ar[d]^{id_{M^*}\widehat{\otimes} \varepsilon^*} & M^* \ar[d]^{=}\\
M^*\widehat{\otimes} \mathcal{O}(G)^*\ar[r]^{\Delta_{M}^*} & M^* & M^*\widehat{\otimes} k\ar[r]^{\cong} & M^*}$$
\end{equation}
\begin{definition}\label{d10}
Let $G$ be an affine group scheme, we say that $(M^*,\Delta_M^*)$ is a right topological $\mathcal{O}(G)^*$-module over the complete Hopf algebra $\mathcal{O}(G)^*$ if $M^*$ is a linearly-compact topological vector space with an action of $\mathcal{O}(G)^*$, that is, $M^*$ and continuous $k$-linear mapping $\Delta_M^*$ are included in the associativity and unit diagrams \eqref{eq8} or equivalently $(M^{\vee}, \Delta_M^{*\vee})$ is a right $\mathcal{O}(G)$-comodule.
\end{definition}

We say that a continuous $k$-linear mapping $\phi:M_1^*\rightarrow M_2^*$ is a homomorphism of right topological $\mathcal{O}(G)^*$-modules if the following diagram is commutative
$$\xymatrix{M_1^*\widehat{\otimes} \mathcal{O}(G)^* \ar[r]^(0.6){\Delta_{M_1}^*}\ar[d]^{\phi\widehat{\otimes} Id}  & M_1^*\ar[d]^{\phi} \\
M_2^*\widehat{\otimes} \mathcal{O}(G)^*\ar[r]^(0.6){\Delta_{M_2}^*} & M_2^*}.$$
By Lemma \ref{l2} below, any comodule $M$ over $\mathcal{O}(G)$ can be decomposed into a direct limit $M\cong\varinjlim M_{\lambda}$ of finite dimensional ``rational'' comodules $M_{\lambda}$ over $\mathcal{O}(G_{\lambda})$, where $\mathcal{O}(G_{\lambda})$ are finitely generated (as algebras) Hopf subalgebras of $\mathcal{O}(G)$. By duality, any topological $\mathcal{O}(G)^*$-module $M^*$ can be decomposed into the inverse limit $M^*\cong\varinjlim M^*_{\lambda}$ of finite dimensional topological $\mathcal{O}(G_{\lambda})^*$-modules.

Our concept generalizes Richard Hain's notion \cite[Theorem 3.4]{Hain1992}, topological modules form abelian category since this category is dual to the category of $G$-modules, which is known to be abelian ~\cite[I,~2.9]{Jan}.


\section{Prounipotent groups and their presentations}\label{s3.0}
The purpose of this section is to introduce prounipotent groups, their presentations and prounipotent crossed modules.

After Definition \ref{d5} of a prounipotent group in Section \ref{s1.2} we give examples.
In addition to the survey of prounipotent groups in zero characteristics in Example \ref{e1.1}, we discuss in Example \ref{e3} free prounipotent groups and their Hopf algebras. Example \ref{e6} explains how to look at finite $p$-groups and pro-$p$-groups as at prounipotent affine group schemes over $\mathbb{F}_p$, we also extract information about the algebra of regular functions of the free pro-$p$-group in Example \ref{e4}. An example of computations cohering zero and positive characteristic is discussed in Remark \ref{r2}.

In Section \ref{s3.4}, the duality of fixed points of $ G $ -modules and coinvariants of their dual topological modules is proven.

Presentation theory of prounipotent groups is contained in Sections \ref{s2}-\ref{s2.2}. The specificity of our approach lies in the fact that we consider not only ``proper'' presentations. On simplicial presentations we impose a non-burdensome constraint on the choice of the defining relations.

The concept of a prounipotent crossed module of a prounipotent presentation is introduced in Section \ref{s3.3}, this Section also contains the formulation of the main result.

Besides homotopy theory importance prounipotent groups are especially cute since they have presentations theory perfectly similar to the discrete case, modules over such affine group schemes are simple enough to get deep results of cohomological nature. We refer the reader to \cite{LM}, \cite{LM2}, \cite{Mag} for pioneering papers on prounipotent groups presentations, to classical results on presentations of pro-$p$-groups \cite{Se4}, \cite{Koch}, \cite{ZR}, \cite{Por}, recent result of the author scattered across the articles \cite{Mikh2016}, \cite{Mikh2017}, \cite{Mikh2017a}, \cite{Mikh2020}.
\subsection{Prounipotent groups}\label{s1.2}

\begin{definition} \label{d5} A unipotent group over a field $ k $ is an affine algebraic group scheme $ G $ for which
every nonzero linear representation $ V $ has a nonzero fixed vector $v\in V$, we should call $v$ fixed if $G$ acts trivially on the subspace $kv$, i.e. $\rho(v)=v\otimes 1$ in $V$, when $k=\mathbb{F}_p$ we restrict ourselves to considering constant affine group schemes of finite $ p $-groups \cite[2.3]{Wat}.

An affine group scheme $G$ is called a prounipotent group if there is a decomposition of $G$ into an inverse limit $G=\varprojlim G_{\alpha}$ of unipotent groups $G_{\alpha}$.
\end{definition}

\begin{example}[Prounipotent groups in zero characteristics]\label{e1.1}

There is also the well known correspondence between unipotent groups over a field $k$ of characteristics 0
and nilpotent Lie algebras over $k$, which assigns to a unipotent group its nilpotent Lie algebra.
This correspondence is easily extends to the correspondence between prounipotent groups over $k$
and pronilpotent Lie algebras over $k$.
Functoriality of the correspondence enables one, when it is convenient, to interpret
problems on unipotent groups in the language of Lie algebras.
Prounipotent groups are curious in a sense that we can work with them as they are ordinary groups: the image of a closed subgroup of a prounipotent group under a homomorphism of prounipotent groups will be always a closed subgroup \cite[Theorem 15.3]{Wat}, the group of rational points of a factor is isomorphic to the factor of groups of rational points \cite[18.2 (e)]{Wat}.
Theorems on the structure of normal series, nontriviality of the center of a unipotent group
are transferred from the corresponding statements for Lie algebras.
Recall also that $G(k)\cong \mathcal{G}\mathcal{O}(G)^*$ \cite[Prop. 18]{Vez}, where $\mathcal{G}$ is the functor of group-like elements in the complete Hopf algebra $\mathcal{O}(G)^*$.

Let $A$ be a Hopf algebra over a field $k$ of characteristic 0, in which: 1) the product is commutative; 2) the coproduct is conilpotent or, equivalently, coconnected in terminology of \cite[Theorem 8.3]{Wat}. Then, as an algebra, $A$ is isomorphic to a free commutative algebra \cite[Theorem 3.9.1]{Car}.
Therefore, each group of $k$--points $G_{\alpha}(k)= Hom_{Alg_k}(\mathcal{O}(G),k)$ is isomorphic as an algebraic variety to certain $n_{\alpha}$-dimensional affine space $\mathbb{A}^{n_{\alpha}}_k$ \cite[Theorem 4.4]{Wat} and hence it is an affine algebraic group, and consequently \cite[Corollary 4.4]{Wat} a linear algebraic group. Thus, we freely use results and methods from the theory of linear algebraic groups in characteristic 0.
\begin{remark}\label{r6}
Let $G$ be a prounipotent group and  let $\widetilde{G}\subset G(k)$ be a discrete finitely generated Zariski dense subgroup of $G(k)$, then $\widetilde{G}^{\wedge}_u\cong G$. Indeed, by Quillen's formulae (Example \ref{e7}) $\mathcal{O}(\widetilde{G}^{\wedge}_u)^*\cong\widehat{k}\widetilde{G}$. Since group-like elements are linearly independent in $CHA$ (the proof is the same as \cite[Theorem 2.1.2]{Abe}) and $\cap \widehat{I}^m=0$, where $\widehat{I}$ is the augmentation ideal of $\mathcal{O}(G)^*$, we have that $\mathcal{O}(\widetilde{G}^{\wedge}_u)^*\subseteq \mathcal{O}(G)^*$. By the universal property of $\widetilde{G}^{\wedge}_u$, $\widetilde{G}^{\wedge}_u(k)$ is embedded as a closed subgroup of $G(k)$ and therefore $\widetilde{G}^{\wedge}_u(k)\cong G(k)$.
\end{remark}

\end{example}

\begin{example}[Regular functions of free $k$-prounipotent groups $char (k)=0$]\label{e3}

The concept of a free prounipotent group of finite rank, say $n$, over a field of zero characteristics is well known and is beautifully presented by Vezzani \cite{Vez}. Its algebra of regular functions is the tensor algebra
 $T(V)$ of a $k$-vector space $V, dim_k(V)=n$ with
the so called \emph{deconcatenation coproduct} defined by the rule $\Delta(e_I)=\Sigma_{JK=I}e_J\otimes e_K$ on products $e_I=e_{i_1}\cdot...\cdot e_{i_k}$,
where $I=(i_1,...,i_k)$ - a multi-index and $e_{i_j}\in Z$, $Z$ is a chosen basis of $V$.

The so called \emph{shuffle product} defines a commutative composition law in $T(V)$, that is defined on words of lengths $m$ and $n$ as a sum over
$(m+n)!/m!n!$ ways of interleaving the two words, as shown in the following examples:

$ab \circ xy = abxy + axby + xaby + axyb + xayb + xyab,$
$aaa \circ aa = 10aaaaa.$
It may be defined inductively by formulas:
$u \circ \omega = \omega \circ u = u,$
$ua \circ vb = (u \circ vb)a + (ua \circ v)b,$
where $\omega$ is the empty word, $a$ and $b$ are single elements, and $u$ and $v$ are arbitrary words.

The shuffle product, the deconcatenation coproduct, obvious unit and counit maps and the universal enveloping algebra antipod \cite[Example 1.8]{Abe} endow $T(V)$ with the structure of a commutative conilpotent Hopf algebra and its dual coincide with the algebra of non-commutative power series in $n$ variables $k\langle\langle X_1,...,X_n\rangle\rangle$.
There are isomorphisms of CHA's \cite[Example 2.11, p.271]{Qui5}
$$\widehat{k}\Phi(Z)\xrightarrow{\phi} k\langle\langle X_i\rangle\rangle_{i\in Z}\xleftarrow{\theta}\widehat{\mathcal{U}}LV, \theta(s)=X_s, \phi(s)=e^{X_s}=\sum_{n=0}^{\infty} X_s^n/n!, s\in Z,$$
where $\widehat{\mathcal{U}}LV$ is the completion of the universal enveloping algebra of a free Lie algebra $LV$ on $V$ with respect to the augmentation ideal and $\widehat{k}\Phi(Z)$ is the completion of the group algebra $k[\Phi(Z)]$ of the discrete free group $\Phi(Z)$ with respect to the augmentation ideal.
\end{example}

\begin{example}[Finite $p$-groups as unipotent groups $char(k)=p>0$]\label{e6}

Let $G$ be a finite $p$-group, then the group algebra $\mathbb{F}_p[G]$ is a cocommutative Hopf algebra with coproduct defined on $g\in G$ by the rule $\Delta (g)=g\otimes g$, antipod $s(g)=g^{-1}$, obvious unit, and the augmentation as counit.

We set $\mathcal{O}(G)^*=\mathbb{F}_p[G]$ and it is easy to check that $\mathcal{O}(G)\cong(\mathcal{O}(G)^*)^{\vee} \cong \mathbb{F}_p^G$ is the commutative Hopf algebra of functions from $G$ to $\mathbb{F}_p$ with commutative multiplication given by the rule $(f_1\cdot f_2)(x)=f_1(x)\cdot f_2(x),$ where $f_i\in \mathcal{O}(G)$, $x\in G$ and the (dual) coproduct given on the dual basis
$$\{e_g, g\in G| e_g(h)=1 \mbox{ if } h=g \mbox{ and } e_g(h)=0 \mbox{ if } h\neq g\}$$
by the rule $\Delta (e_g)=\sum_{g_1g_2=g} e_{g_1}\otimes e_{g_2}$, where $g_1, g_2\in G$.
The corresponding affine group scheme \cite[2.3]{Wat} with the algebra of regular functions $\mathcal{O}(G)$ is unipotent.

Indeed, let we are given an arbitrary representation $\rho: M\rightarrow M\otimes \mathcal{O}(G)$, we want to show that there is a subspace $\mathbb{F}_p\cdot v, v\in M$ fixed under the action of $G$ i.e. $\rho(v)=v\otimes 1$. By Lemma \ref{l2} this is equivalent to show that the coinvariants $M^*_G\neq 0$.
 We prove by induction on a rank $n$ of $M^*$. If $n=1$ then $M^*=\mathbb{Z}/p\mathbb{Z}$ is a trivial $G$-module, since $(\mid Aut(\mathbb{Z}/p\mathbb{Z})\mid,p)=(\mid\mathbb{Z}/(p-1)\mathbb{Z}\mid,p)=1$, so
$M^*_G=M^*\neq0$. Let $n=k$, then a submodule of $G$-fixed elements of $M^*$ is not trivial and contains $M^*_0=\mathbb{Z}/p\mathbb{Z}$  with trivial $G$-action.
Let $M^*_1=M^*/M^*_0$ and $\psi:M^*\rightarrow M^*_1$ be the corresponding factorization. Since
$\psi(M^*(g-1))=M^*_1(g-1)$,  $\psi$ induces the surjection $M^*_G \rightarrow (M^*_1)_G$.
But $(M^*_1)_G\neq 0$ by induction hence $M^*_G\neq0$.
It is also clear that $G(\mathbb{F}_p)=Hom_{Alg_k}(\mathcal{O}(G),\mathbb{F}_p)\cong \mathcal{G}\mathcal{O}(G)^*\cong G$.

Passing to the inverse limit of group-valued functors, we prove that pro $p$-groups are naturally $\mathbb {F} _p$-points of prounipotent groups over $\mathbb{F}_p$.

\end{example}
\begin{example}\label{e4}(Regular functions of free pro-$p$-groups $char (k)=p>0$)

Let $F(Z)$ be a free pro-$p$-group over a finite set $Z$.
The complete group ring $\mathbb{F}_pF(Z)$ can be endowed with the structure of a complete Hopf algebra. Consider the decomposition $F(Z)\cong \varprojlim G_{\alpha}$ into the inverse limit of finite $p$-groups $|G_{\alpha}|=p^l$, then $\mathbb{F}_pF(Z)\cong\varprojlim \mathbb{F}_p[G_{\alpha}]$ \cite[7]{Koch}.
Each $\mathbb{F}_p[G_{\alpha}]$ inherits the Hopf algebra structure as in Example \ref{e6}.
Taking the inverse limit $\widehat{\Delta}=\varprojlim \Delta_{\alpha}$ we obtain the coproduct
$\widehat{\Delta}:\mathbb{F}_pF(Z)\rightarrow \mathbb{F}_pF(Z)\widehat{\otimes} \mathbb{F}_pF(Z)$.

By \cite[Proposition 7.16]{Koch}, there is an isomorphism of complete group algebras
$\mathbb{F}_p\langle\langle X_i\rangle\rangle_{i\in Z}\cong \mathbb{F}_pF(Z)$, given on generators $s_i$ of $F(Z)$ by the rule
$s_i\mapsto X_i+1$, $s_i^{-1}\mapsto \sum_{n=0}^{\infty}(-X_i)^n$, $X_i\mapsto s_i-1$.
It is easy to check that on generators $X_i$ of $\mathbb{F}_p\langle\langle X_i\rangle\rangle_{i\in Z}$ the coproduct $\widehat{\Delta}$ is given as $\widehat{\Delta}(X_i)=X_i\widehat{\otimes} X_i+1\widehat{\otimes} X_i+ X_i\widehat{\otimes} 1$.
The continuous dual is also the tensor algebra $\mathbb{F}_p\langle\langle X_i\rangle\rangle_{i\in Z}^{\vee}\cong T(\mathbb{F}_p^Z)$ of the vector space with the basis $e_i, i\in Z$ and the deconcatenation coproduct.
By duality, as in Example \ref{e6}, we get a commutative multiplication in $T(\mathbb{F}_p^Z)$ as follows $(f_1*f_2)(g)=m(f_1\otimes f_2)(\widehat{\Delta}(g))=m(f_1\otimes f_2)(g\otimes g)=f_1(g)\cdot f_2(g)$ and dual antipod.

Set $\mathcal{O}(F(Z))=T(\mathbb{F}_p^Z)$ with just defined Hopf algebra structure.
Since $f_1, f_2$ are continuous functions, there must be some $\alpha,$ that $f_1, f_2$ are factored as composition $f_1,f_2:G\xrightarrow{pr_{\alpha}} G_{\alpha}\xrightarrow{\widetilde{f}_1,\widetilde{f}_2} \mathbb{F}_p$ \cite[Lemma 1.1.16]{ZR}
and, therefore, in the pro $ p $-case, continuous functions are ``regular'' functions from algebraic (finite) factors , embedded by $pr_{\alpha}^*:\mathcal{O}(G_{\alpha})\rightarrow \mathcal{O}(G)$.

\end{example}

\begin{remark}[Cohering  positive and zero characteristic]\label{r2}

It is curious that pro-$p$-groups cohere information between positive and zero characteristics, and the appearance of constant unipotent group schemes over $\mathbb{F}_p$ is predetermined.
Let us illustrate this with an example of calculations that were used in \cite[Proposition 3.19 (proof)]{Mikh2016} to solve specific problems.
Let $F$ be the free prounipotent group of rank $m$ over $\mathbb{Q}_p$, then, using Quillen's formulae, we get
$$\mathcal{O}(F)^*\cong \varprojlim \mathbb{Q}_p[\Phi(m)]/I^n_{\mathbb{Q}_p}\cong \varprojlim \mathbb{Z}_p[\Phi(m)]/I^n_{\mathbb{Z}_p}\otimes_{\mathbb{Z}_p}\mathbb{Q}_p,$$
where $\mathbb{Q}_p[\Phi(m)]$ is the group ring of free discrete group of rank $m$, $I_{\mathbb{Z}_p}$ - the augmentation ideal of $\mathbb{Z}_p[\Phi(m)]$.
It follows from Lazard's theorem \cite[Proposition 7]{Se4}, that $\varprojlim\mathbb{Z}_p[\Phi(m)]/I^n_{\mathbb{Z}_p}\cong \varprojlim A(m)/\mathbb{I}^n\cong\varprojlim \mathbb{Z}_pF_p(m)/\widehat{I}^n_{\mathbb{Z}_p},$ where $\mathbb{Z}_pF_p(m)$ is the complete group ring of the free pro-$p$-group $F_p(m)$ of rank $m$ and $\widehat{I}^n_{\mathbb{Z}_p}$ is the $n$-th power of its augmentation ideal, $A(m)\cong \mathbb{Z}_p\langle\langle m\rangle\rangle$ - the noncommutative power series of $m$ variables over $\mathbb{Z}_p$, $\mathbb{I}$ - its augmentation ideal.
The $mod(p)$-reduction of group rings coefficients
$\varprojlim \mathbb{Z}_pF_p(m)/\widehat{I}^n_{\mathbb{Z}_p}\rightsquigarrow^{mod(p)} \varprojlim \mathbb{F}_pF_p(m)/\widehat{I}^n_{\mathbb{F}_p}\cong \mathcal{O}(F_p(m))^*$ leads to the complete $\mathbb{F}_p$-Hopf algebra of the free pro-$p$-group of rank $m$.

The $\widehat{I}_{\mathbb{F}_p}$-adic filtration induces the so called Zassenhaus filtration of $F_p(m)$ $\mathcal{M}_{n,p}=\{f \in F\mid f-1 \in \widehat{I}_{\mathbb{F}_p}^n, n\in \mathbb{N}\}$ and it follows from \cite[7.4]{Koch} that $F_p(m)/\mathcal{M}_{n,p}\cong \Phi(m)/\Phi(m)\cap\mathcal{M}_{n,p}$ are finite $p$-groups and hence $\mathcal{O}(F_p(m))^*\cong \varprojlim\mathbb{F}_p[\Phi(m)]/I^n_{\mathbb{F}_p}\cong \varprojlim \mathbb{F}_p[\Phi(m)/\Phi(m)\cap\mathcal{M}_{n,p}]\cong \varprojlim \mathcal{O}(\Phi(m)/\Phi(m)\cap\mathcal{M}_{n,p})^*$.
The Hopf algebra structure on $((\mathcal{O}(\Phi(m)/\Phi(m)\cap\mathcal{M}_{n,p})^*)^{\vee}$ is uniquely defined, since there is the only structure of a commutative algebra with unit for the function space of a finite set (analogously to Gelfand's result \cite[11.13(a)]{Ru}) and the ``constant'' comultiplication is given by duality (Example \ref{e6}).
Mimicking Corollary \ref{c2}, we may think that complete group algebras with $\mathbb{Z}_p$-coefficients are $\mathbb{Z}_p$-dual to the distribution algebras \cite[Proposition 1]{Mikh2020}.
\end{remark}

\subsection{Fixed points and coinvariants}\label{s3.4}

From now on $G\cong\varprojlim G_{\beta}$ be a prounipotent group decomposed into an inverse limit of unipotent groups and let $M$ be a right topological $\mathcal{O}(G)^*$-module, denote by $M_G=M/M(g-1)$ - the module of $G$-coinvariants. This is a quotient of $ M $ by a closed submodule generated by elements $\{m(g-1), g\in G(k), m\in M\}$.

\begin{lemma}\label{l2}
Let $G\cong \varprojlim G_{\beta}$ be a prounipotent group and $V$ be a $G$-module, then $V$ is decomposed as $V=\varinjlim V_{\alpha},$ where $V_{\alpha}$ are $G_{\alpha}$-modules and $G_{\alpha}$ could be chosen from the given decomposition $G\cong \varprojlim G_{\beta}$. Dually, there is a decomposition $V^*=\varinjlim V^*_{\alpha}$ into the inverse limit of finite dimensional $\mathcal{O}(G_{\alpha})^*$-modules.
In particular $V=0$ if and only if $V^G=0$, or equivalently $V^*=0$ if and only is $(V^*)_G=0$.
\end{lemma}
\begin{proof}
According to \cite[3.3, Theorem (1)]{Wat} any finite subset $X$ of $V$ lies in a $k$-subspace of finite dimension $W\subseteq V$ such that $\Delta(W)\subset W\otimes \mathcal{O}(G)$.
Let $\{w_i\}$ be a basis of $W$ and consider $\Delta(w_i)=\sum w_j\otimes a_{ij}$. We fix the basis $\{v_k\}$ of $\mathcal{O}(G)$ and let $\Delta(a_{ij})=\sum v_k\otimes a_{ijk}$, then \cite[3.3, Theorem (2)]{Wat} says that $k[a_{ij},a_{ijk},Sa_{ij},Sa_{ijk}]\subseteq\mathcal{O}(G)$ is a finitely generated (as algebra) Hopf subalgebra of $ \mathcal{O}(G)$.

It is the intrinsic property of a unipotent group to have a fixed point in any nontrivial representation and therefore $V=0$ if and only $V^G=0$, the rest follows from Lemma \ref{l4} below.
\end{proof}
\begin{lemma}\label{l4} Let $M$ be a $\mathcal{O}(G)$-comodule over the Hopf algebra of regular functions of a prounipotent group $G$ and let $M^*=Hom_k(M,k)$ be the dual topological $\mathcal{O}(G)^*$-module. Then linearly compact topological spaces $(M^G)^*\cong (M^*)_G$ are naturally isomorphic.
\end{lemma}
\begin{proof} We prove in the case of zero characteristics, the pro-$p$-case is analogous. We start with a short exact sequence of \eqref{eq2}
$0\rightarrow M^G\xrightarrow{i} M\otimes k \xrightarrow{\Delta_M-id_M\otimes 1} M\otimes \mathcal{O}(G),$
where $i:M^G\rightarrow M\cong M\otimes k$ is the embedding $i(m)=m\otimes 1$.
Consider the dual exact sequence of linearly-compact topological $k$-vector spaces

$$\xymatrix{0& (M^G)^* \ar[l]& (M\otimes k)^* \ar[l]_{i^*}&& M^*\widehat{\otimes}\mathcal{O}(G)^* \ar[ll]_{(\Delta_M-id_M\otimes 1)^*}}.$$

We have a decomposition of topological vector spaces $\mathcal{O}(G)^*=k\cdot 1+ I$ and, therefore, $Im\;(\Delta_M-id_M\otimes 1)^*=M^*\cdot I$, where $I=Ker\; (\mathcal{O}(G)^*\rightarrow k)$ - the augmentation ideal.
Lemma \ref{l2} gives the decomposition $M^*\cong \varprojlim M^*_{\alpha}$, where $M^*_{\alpha}$ are finite dimensional topological $\mathcal{O}(G_{\alpha})^*$-modules (in the notations of Lemma \ref{l2}) and, therefore, $M^*\cdot I\cong \varprojlim (M^*_{\alpha}\cdot I_{\alpha})$, where $I_{\alpha}$ - the augmentation ideal of $\mathcal{O}(G_{\alpha})^*$.
Now from Quillen's formulae (Examples \ref{e7}, \ref{e1.1}) it follows that $M^*_{\alpha}\cdot I_{\alpha}$ is generated by the elements $\{m\cdot g-m|m\in M_{\alpha}, g\in G_{\alpha}(k)\}$ and the result follows.
\end{proof}

\subsection{Presentations of prounipotent groups}\label{s2}
In Section \ref{s1.2} we saw that a prounipotent group, both in zero characteristics and in the pro-$p$-case, can be defined as group-valued functors representable by means of commutative Hopf algebras.
The notion of a free prounipotent group $F(Z)$ over a finite set $Z$ is well known and nicely explained, as mentioned above in \cite[3]{Vez}, \cite{HM2003}.
The group of rational points $F(Z)(k)$ posses a universal property with respect to mappings of $Z$ into rational points of prounipotent groups, since it is the prounipotent completion of the free discrete group $\Phi(Z)$ on $Z$.
 We extend the notion of a free prounipotent group to an arbitrary set $Z$, in the case of an algebraically closed ground field of zero characteristics it coincides with \cite[Definition 2.1]{LM}.

 Let $Z$ be a set and let $\{Z_i|i\in J\}$ be the collection of all finite subsets of $Z$. Make $J$ into a poset by defining $i\preceq j$
if $Z_i\subseteq Z_j$. If $i\preceq j$ define $\psi^j_i:\mathcal{O}(F(Z_i))\rightarrow \mathcal{O}(F(Z_j))$ as the tensor algebras embedding of Examples \ref{e3}, \ref{e4}, induced by the inclusion of finite sets
$\psi^j_i:Z_i\rightarrow Z_j$, recall that in both cases, in the case of $ char (k) = 0 $ and in the pro-$p$-case, we have, as vector spaces, $\mathcal{O}(F(Z_i)(k))=T(k^{Z_i})$.
We define the Hopf algebra of regular functions $\mathcal{O}(F(Z))$ of a free prounipotent group over an arbitrary set $Z$ as
$$\mathcal{O}(F(Z))= \varinjlim_{i\in J}\mathcal{O}(F(Z_i)).$$
In the pro-$p$-case, for $f_1\in \mathcal{O}(F(Z_1)), f_2\in \mathcal{O}(F(Z_2))$ their product $f_1\cdot f_2$ is properly defined as an element of $\mathcal{O}(F(Z_1\cup Z_2))$ as in Example \ref{e4}.

\begin{definition} \label{d006} By a free $k$-prounipotent group ($F(Z), char (k)=0$) over a set $Z$ we assume a $k$-prounipotent group with a Hopf algebra of regular functions $\mathcal{O}(F(Z))\cong T(k^Z)$ isomorphic to the tensor algebra of a vector space $k^Z$ with the shuffle product multiplication, the deconcatenation coproduct and the universal enveloping algebra antipod.

By a free pro-$p$-group ($F(Z), k=\mathbb{F}_p$) over a set $Z$ we assume a $\mathbb{F}_p$-prounipotent affine group scheme with a Hopf algebra of regular functions $\mathcal{O}(F(Z))\cong T(k^Z)$ isomorphic to the tensor algebra of a vector space $k^Z$ with the product of functions as their multiplication, the deconcatenation coproduct and antipod, induced using duality by the group inversion.
\end{definition}

\begin{lemma}\label{l7} Let $X,Y$ be sets, then $F(X\cup Y)$ is the direct sum of $F(X)$ and $F(Y)$: $F(X\cup Y)\cong F(X)\star F(Y)$.
\end{lemma}
\begin{proof} By Yoneda's Lemma we prove, that $\mathcal{O}(F(X\cup Y))$ with natural projections $p_1:\mathcal{O}(F(X\cup Y))\rightarrow \mathcal{O}(F(X)), p_2:\mathcal{O}(F(X\cup Y))\rightarrow \mathcal{O}(F(Y))$ is the direct product of $\mathcal{O}(F(X))$ and $\mathcal{O}(F(Y))$, i.e. we must show, that for any $A=\mathcal{O}(G),$ where $G$ is a prounipotent group, and any Hopf algebra homomorphisms $\phi_1:A\rightarrow \mathcal{O}(F(X))$ and $\phi_2:A\rightarrow \mathcal{O}(F(Y))$ there is a unique $\psi: A\rightarrow \mathcal{O}(F(X\cup Y))$, such that $p_1=i_1 \psi$ and $p_2=i_2 \psi$. Since $A\cong \varinjlim A_{\alpha},$ where $A_{\alpha}$ are finitely generated (as algebras) Hopf algebras \cite[Theorem 3.3]{Wat} and since $\phi_1(A_{\alpha})\subseteq T(k^{X_{\alpha}})$, $\phi_2(A_{\alpha})\subseteq T(k^{Y_{\alpha}})$, where $X_{\alpha}\subseteq X$, $Y_{\alpha}\subseteq Y$ are finite subsets of $X,Y$, then $F(X_{\alpha}\cup Y_{\alpha})\cong F(X_{\alpha})\star F(Y_{\alpha})$ by the universal property of finitely generated free prounipotent groups, which holds since they are prounipotent completions of finitely generated free discrete groups on the same generating sets, therefore
there is a unique $\psi_{\alpha}:A_{\alpha}\rightarrow T(k^{X_{\alpha}\cup Y_{\alpha}})$ with the required property, and we define $\psi=\varinjlim \psi_{\alpha}$.
\end{proof}

\begin{remark}[Duality between tensor algebra generators and free prounipotent group generators]\label{e5}
Let $Z$ be a finite set and $\mathcal{O}(F(Z))\cong T(k^Z)$ be the Hopf algebra of regular functions of a free prounipotent (pro-$p$-group) group $F(Z)$. It follows from definition of the deconcatenation coproduct that a subspace of primitive elements $\mathcal{P}\mathcal{O}(F(Z))=\{x\in \mathcal{P}\mathcal{O}(F(Z))| \Delta(x)=1\otimes x + x\otimes 1, \varepsilon(x)=0\}$, where $\varepsilon$ is the counit, coincides with $k^Z$.

On the other hand $\mathcal{P}\mathcal{O}(F(Z))\cong C_1/C_0,$ where $C_k=Ann_{\mathcal{O}(F(Z)} I^{k+1}$ is the annihilator of $(k+1)$-th power of the augmentation ideal $I$ of $\mathcal{O}(F(Z))^*$, as
$C_0=Ann_{\mathcal{O}(F(Z)} I=k\cdot 1$ and $C_1=k\cdot 1\oplus \mathcal{P}\mathcal{O}(F(Z))$. Indeed, $x\in Ann_{\mathcal{O}(F(Z)} I^2$ if and only if $\Delta(x)=y\otimes1+ 1\otimes z; y, z\in \mathcal{O}(F(Z))$ ($x(i_1\cdot i_2)=m(i_1\otimes i_2)\Delta(x)$ and $1\cdot k$ is the annihilator of $I$), it follows from Hopf algebra axioms, that this is equivalent if $\epsilon(x)=0$ to $\Delta(x)=1\otimes x+1\otimes x$.
By duality, $Ann_{\mathcal{O}(F(Z))} I^k\cong (\mathcal{O}(F(Z))/I^k)^{\vee}$ and hence
$$I/I^2\cong Ker\;(\mathcal{O}(F(Z))/I^2\rightarrow \mathcal{O}(F(Z))/I)\cong (C_1/C_0)^*\cong (\mathcal{P}\mathcal{O}(F(Z))^*$$
and therefore  $I/I^2\cong (k^Z)^*$.

It follows from \cite[7.4, Lemma 4.10 (proof (I)]{Koch} that for a free pro-$p$-group $F=F(Z)$ the formulae $(X_i-1) +I^2\mapsto s_i\cdot[F,F]F^p, s_i\in Z$ gives the isomorphism $I/I^2\cong F/[F,F]F^p$.
We can also deduce From Example \ref{e3} that $I/I^2\cong F/[F,F](k)$ for a free prounipotent group over the base field $k$ of zero characteristic, where ,in both cases, $I$ is the augmentation ideal of $\mathcal{O}(F(Z))^*$.

By \cite[Proposition 25]{Se4} in the pro-$p$-case and by \cite{LM}, when $char(k)=0$, the sets of Zariski generators of $F(Z)(k)$ are in one-to-one correspondence  with the lifts of the bases of $F/F^p[F,F]$, in the pro-$p$-case, and of $F/[F,F](k)$, when $char(k)=0$, under the abelianization homomorphism $F\rightarrow F/F^p[F,F]$, in the pro-$p$-case, and $F(k)\rightarrow F/[F,F](k)$, $char(k)=0$.
As a consequence we have the right to consider the generators of the free prounipotent (pro-$p$) group $F(Z)(k)$ as elements of $Z^*$ dual to the basis $Z$ of $k^Z$ which generate $T(k^Z)\cong \mathcal{O}(F(Z))$. This lifting can be visualized using Quillen's formula, since $\Phi(Z)\subset F(Z)(k)$, as in Example \ref{e1.1}, and $\Phi(Z)$ maps to $\Phi(Z)/[\Phi(Z),\Phi(Z)]$ under the abelianization homomorphism $F(k)\rightarrow F/[F,F](k)$ \cite[Appendix A.1.]{HM2003}.
\end{remark}

In order to characterize free prounipotent groups as groups with a universal property (Definition \ref{d11} and Remark \ref{r4} below) we introduce an inverse system of pointed finite sets which arises by duality of Remark \ref{e5} from the direct system of finite subsets of $Z$,
this will be convenient when working with free topological modules in Section \ref{s3.1}.
\begin{remark}[Pointed profinite spaces of generators]\label{r5}
Let $Z$ be a set and let $\{Z_i|i\in J\}$ be the collection of all finite subsets of $Z$. Make $J$ into a poset by defining $i\preceq j$
if $Z_i\subseteq Z_j$. If $i\preceq j$ we define $\phi^j_i:F(Z_j^*,*)\rightarrow F(Z_i^*,*)$ as a homomorphism of free finitely generated prounipotent groups induced by the map of punctured finite spaces $\phi^j_i:(Z_j^*,*)\rightarrow (Z_i^*,*)$ (as in the Remark \ref{e5} we consider the generators of $Z_i^*\hookrightarrow F(Z_i)(k)$ as a dual basis of $(k^{Z_i})^*$, where $\mathcal{O}(F(Z_i)(k))=T(k^{Z_i})$)
\begin{equation}\label{eq3}
\phi^j_i(z)=z, \mbox{if } z\in Z_i^*, \phi^j_i(z)=\{*\}, \mbox{ if } z\in Z_j^* - Z_i^*, \phi^j_i(*)=\{*\}.
\end{equation}
We always consider $(Z_i^*,1)$ as the subspace of rational points $F(Z_i)(k)$, and we also assume that the specified space is nested $(Z_i^*,*)\hookrightarrow F(Z_i)(k)$ by the rule $z \mapsto z, z\in Z_i$ and $*\mapsto 1\in F(Z_i)(k)$, where $i\in J$.

Left exactness of $\varprojlim$ implies that $(\widehat{Z}^*,*)=\varprojlim_{i\in J} (Z_i^*,*)$ is actually a profinite subspace of $F(Z)(k)$ (this subspace is actually a one-point compactification of $Z$). We will write $F=F(\widehat{Z}^*,*)$, emphasizing that $F(k)$ is generated by $(\widehat{Z}^*,*)$.
\end{remark}


Now we show that the group of Definition \ref{d006} possess a universal property with respect to the ``convergent'' maps of the ``pointed profinite basis'' $(\widehat{Z}^*,*)$ of $F(\widehat{Z}^*,*)(k)$ to rational points $G(k)$ of the unipotent group $G$.

\begin{proposition}\label{p001}
Let $G$ be a unipotent group, then there is a bijection between
homomorphisms $f:F(\widehat{Z}^*,*)\rightarrow G$ and continuous maps $f':(\widehat{Z}^*,*)\rightarrow G(k)$, such that $f'(*)=e$ and  $f'(\widehat{Z}^*,*)$ has finite cardinality.
\end{proposition}
\begin{proof}
Indeed, suppose we are given a homomorphism $f:F(\widehat{Z}^*,*)\rightarrow G$ of a free prounipotent group $F(\widehat{Z}^*,*)$ into a unipotent group $G$. Since $\mathcal{O}(G)$ is a finitely generated Hopf algebra it follows that the image $f^o(\mathcal{O}(G))$ lies in the Hopf subalgebra $\mathcal{O}(F(Z_i^*,*))=T(k^{Z_i})$ for some finite subset $Z_i\subseteq Z$.
$\mathcal{O}(F(Z))$ is faithfully flat over $\mathcal{O}(F(Z_i^*,*))$ \cite[Theorem 14.1]{Wat} and therefore \cite[Theorem 13.2]{Wat} $f$ is factored via an epimorphism onto a free prounipotent group over $Z_i^*$, which corresponds to a continuous mappings of pointed profinite spaces $g':(\widehat{Z}^*,*)\rightarrow (Z_i^*,*)$ and, therefore, taking the composition with the homomorphism $h':F(Z_i^*,*)\rightarrow G(k)$ we obtain the required $f'=h'g'$.

If we are given a continuous map $f':(\widehat{Z}^*,*)\rightarrow G(k)$ with a finite image then it factors through some projection $\phi_i:(\widehat{Z}^*,*)\rightarrow (Z^*_i,*)$ \cite[Lemma 1.1.16]{ZR}. This projection corresponds to the homomorphism of prounipotent groups $\phi_i:F(\widehat{Z}^*,*)\rightarrow F(Z^*_i,*)$ as above. But $F(Z^*_i,*)$ is the prounipotent completion of $\Phi(Z_i^*)$ - the free discrete group on $Z_i^*$ \cite[3]{Vez}. Now the result follows from the universal properties of the free discrete group over $Z_i^*$ and its prounipotent completion.
\end{proof}
\begin{remark}\label{r8}
It turns out that the decomposition $F(\widehat{Z}^*,*)\cong \varprojlim F(Z_i^*,*)$ is actually an inverse limit of proper closed subgroups, analogous statement is true for the pro-$\mathcal{C}$-case \cite[Corollary 3.3.10]{ZR}.
Indeed, the natural maps of pointed profinite spaces $i:(Z_i^*,*)\rightarrow (\widehat{Z}^*,*)$ and $pr_i:(\widehat{Z}^*,*)\rightarrow (Z_i^*,*)$ induce, by the universal property, the homomorphisms of free prounipotent groups $i:F(Z_i^*,*)\rightarrow F(\widehat{Z}^*,*)$, $pr_i:(\widehat{Z}^*,*)\rightarrow (Z_i^*,*)$ and since $pr_i\circ i=id_{F(Z_i^*,*)}$ then $i$ establishes an isomorphism of the group $F(Z_i^*,*)$ with a closed subgroup in $F(\widehat{Z}^*,*)$ generated by $(Z_i^*,*)$ and, therefore, $pr_i, i\in I$ are actually projections onto closed proper subgroups, which determined by the Hopf ideals of $\mathcal{O}(F(X))$ generated by $k^{X\setminus X_i}$. As a consequence, once chosen generators in $F(Z_i^*,*)(k)$ are actually generators in $F(\widehat{Z}^*,*)(k)$ as expected.
\end{remark}

Proposition \ref{p001} shows that free prounipotent groups in Definition \ref{d006} coincide with the Lubotzky-Magid's \cite[Definition 2.1]{LM}. These groups can be characterized as groups with the so called \emph{lifting property}:
$$\xymatrix{1 \ar[r] & K \ar[r] & E \ar[r]^{\pi} & G \ar[r] & 1 \\
  && F \ar[u]^{h} \ar[ur]_{\rho} \\ }$$
for any exact sequence of prounipotent groups and homomorphism $\rho:F\rightarrow G$ there is a homomorphism $h:F\rightarrow E$ such that $\pi h=\rho.$

As checked in \cite[Theorem 2.4]{LM} the lifting property characterize free pro-unipotent groups over algebraically closed fields of zero characteristics. However, the proof depends only on standard results on normal subgroups of unipotent groups, which are easy to follow from the corresponding results on nilpotent Lie algebras \cite{Se}, thus free prounipotent groups over nonalgebraically closed fields of characteristic zero are also characterized as groups with the lifting property. In fact, in the pro-$p$-case the indicated characterizations of free pro-$p$-groups were obtain back in \cite[3.4]{Se4}.

\begin{lemma}\label{l1234}
Let $F$ be a free prouniotent group, then any epimorphism $\phi:G\rightarrow F$ of prounipotent groups has a splitting, which is a homomorphism of prounipotent groups.
\end{lemma}
\begin{proof}
Take $G=F, E=G$ and $\rho=id_F$ in the diagram of lifting property above, then $h$ is the desired splitting.
\end{proof}

We say that the prounipotent group $G$ has \emph{cohomological dimension} $n$, and write $cd(G)=n$, if for every $G$-module $A$ and every $i>n$, $H^i(G,A)=0$ and $H^n(G,V)\neq 0$ for some $G$-module $V$.
Prounipotent groups of cohomological dimension one are free prounipotent groups,
this follows from Proposition \ref{p001} and \cite[(1.18)]{LM}.
An important feature of prounipotent groups (which makes their cohomological study sometimes easier than in the discrete case) is that $cd(G)\leq n$ if and only if $H^{n+1}(G,k_a)=0$, because $k_a$ is the only simple $G$-module. This fact is easily deduced with the ``d\'{e}vissage'' argument \cite[3.1]{Se4}.

The following propositions makes possible to introduce a concept of pro-unipotent presentation in the form \eqref{eq1}.

\begin{proposition}\label{p7} Let $G$ be a prounipotent group over a base field $k$ of zero characteristics or $k=\mathbb{F}_p$ and $G$ is a pro-$p$-group. Then there is a free prounipotent group $F(X)$ over a set $X$ of cardinality less or equals a cardinality $d(G)$ of a basis of $Hom(G,k_a)$ and an epimorphism $f:F(X)\rightarrow G$.
\end{proposition}
\begin{proof} The proof uses the lifting property of free prounipotent groups and Lemma \ref{l1234} (see details \cite[Proposition 2.8]{LM} which are similar to the pro-$p$-case).
\end{proof}
Cohomology and ``proper'' presentation theory of prounipotent groups in zero characteristics closely parallels that of the cohomology of pro-$p$-groups \cite{Se4}: the free prounipotent groups turn out to be those of cohomological dimension ones, the dimension of the first and second cohomology groups give cardinalities of generators and defining relations \cite{LM}, \cite{LM2}.

\begin{definition}\label{d11}
We shall say that a map $\mu: (\widehat{X}^*,*)\rightarrow G(k)$ of a pointed profinite space $(\widehat{X}^*,*)$ into the group of rational points of a prounipotent group $G$ is \textbf{convergent} if there is a decomposition $G(k)\cong \varprojlim G_{\alpha}(k)$ such that for all but finite $x\in \widehat{X}^*$ we have that $\phi_{\alpha}\mu(x)=1$, where $\phi_{\alpha}:G\rightarrow G_{\alpha}$ is the projection.
\end{definition}
\begin{remark}\label{r4}
By construction, free prounipotent groups have the universal property with respect to convergent maps of generators into rational points of prounipotent groups. Let $f:F(\widehat{X}^*,*)(k)\rightarrow G(k)$ be a homomorphism from a free prounipotent group over a pointed profinite space $(\widehat{X}^*,*)$ onto $k$-points of the prounipotent group $G$, then, according to the arguments of the proof of Proposition \ref{p001}, the map $(\widehat{X}^*,1)\rightarrow G(k)$ is convergent.
\end{remark}
\subsection{Simplicial presentations of prounipotent groups}\label{s2.1}
In \cite{Mikh2016}, \cite{Mikh2017} the notion of a prounipotent presentation of finite type was extended to prounipotent groups over non-algebraically closed fields of characteristic zero. Their simplicial forms also appeared in \cite{Mikh2016}, \cite{Mikh2017} as free simplicial prounipotent groups of finite type, degenerate in dimensions greater than one. In this section we define the notion of a simplicial prounipotent presentation over punctured profinite spaces.

Our exposition is completely analogous in zero characteristics and in the pro-$p$-case, so we restrict ourselves by the first case. The only difference is that we need to replace $[R,R]$ with $R^p[R,R]$ and $[R,F]$ with $R^p[R,F]$ to work with $\mathbb{F}_p$-vector spaces in the pro-$p$-case.

\begin{remark}\label{r9}
By Proposition \ref{p7}, any prounipotent group $G$ can be presented as a quotient of a free pro-unipotent group $G\cong F(X)/R=F(\widehat{X}^*,*)/R$ for some set $X$ of cardinality less than or equal to the cardinality of the basis in $Hom(G,k_a)$.
If the inflation map $H^1(G,k_a)\rightarrow H^1(F,k_a)$ is an isomorphism, then we say that such a prounipotent presentation is proper.
In what follows, we will not restrict ourselves only to proper representations, but will consider only those sets of ``defining relations'' $Y$ for which there is a fixed isomorphism $(R/[R,F](k))^{\vee}\cong k^Y$ or dually  $R/[R,F]\cong (k^Y)^*\cong k^{(\widehat{Y}^*,*)}$, we will explain this condition in Remark \ref{r10} below.
\end{remark}
We will call $1\rightarrow R\rightarrow F(\widehat{X}^*,*)\rightarrow G\rightarrow 1$ a prounipotent presentation $(X|R)$ in the form \eqref{eq1} and construct its simplicial form into three steps:

Step 1. Choose a basis, say $Y$, in $R/[R,F](k)^{\vee}$, then $R/[R,F](k)\cong (k^Y)^*\cong k^{(\widehat{Y}^*,*)}$, by duality, and $\tau_F:(\widehat{Y}^*,*)\rightarrow (\widehat{Y}^*,0)\subset R/[R,F](k)$ is by definition a homeomorphism of punctured profinite spaces, which sends $y\mapsto y, y\in \widehat{Y}^*$ and $*\mapsto 0$.

Step 2. Let $\psi: R/[R,R](k)\rightarrow R/[R,F](k)$ be a projection of linearly-compact topological vector spaces induced by the inclusion $[R,R]\subset [R,F]$, then the dual inclusion map $\psi^{\vee}: R/[R,F](k)^{\vee}\rightarrow R/[R,R](k)^{\vee}$ allows us to consider the basis $Y\subset R/[R,F](k)^{\vee}$ as a subset of some basis $T$ of $R/[R,R](k)^{\vee}$, i.e. $Y\subset T$. We denote
$\sigma:R/[R,R](k)^{\vee}\rightarrow R/[R,F](k)^{\vee}$ - the projection onto the subspace, which is given on the basis by the rule $\sigma(x)=x, x\in Y$ and $\sigma(x)=0$, $x\notin Y$ and, by duality, we obtain the closed embedding $\sigma^*:R/[R,F](k)\rightarrow R/[R,R](k)$ which is a continuous splitting of $\psi$ and, in particular, $(\widehat{T}^*,0)=(\widehat{Y}^*,0)\cup (\widehat{T\setminus Y }^*,0)$.

Step 3. Let $T$ be a basis of $R/[R,R](k)^{\vee}$, by Remark \ref{r5} there exists a lifting homeomorphism of pointed profinite spaces $\eta: (\widehat{T}^*,0)\rightarrow R(k)$ and the image $\eta(\widehat{T}^*,0)$ Zariski generates $R(k)$. We define a pointed profinite space of ``defining retations'' as the image of
\begin{equation}\label{eq9}
\tau=\eta\circ\sigma^*\circ\tau_F: (\widehat{Y}^*,*)\rightarrow R(k).
\end{equation}
We claim that $\tau(\widehat{Y}^*)$ Zariski generates $R(k)$ as a normal subgroup of $F(\widehat{X}^*,*)$. The proof of this statement is quite standard \cite[Theorem 3.11]{LM}, \cite[Proposition 7.8.2]{ZR} and we recall it.
We have to prove that $N(k)=\langle\widehat{Y}^*\rangle_F$, i.e. the Zariski normal closure of $\widehat{Y}^*$ coincides with $R(k)$.
Indeed, the inclusion $N\subseteq R$ gives the restriction homomorphism $Res:H^1(R,k_a)\rightarrow H^1(N,k_a)$, which induces (by taking $F$-fixed points) $Hom_F(R,k_a)=H^1(R,k_a)^F\xrightarrow{Res^F} H^1(N,k_a)^F$ a monomorphism (by construction of $N$, see also \cite[Proposition 26]{Se4}), i.e. $Ker(Res)^F=0$ and, therefore, by Lemma \ref{l2} $Ker(Res)=0.$
Now suppose that $N\neq R$, then there is a non trivial homomorhism $\xi:R/N\rightarrow k_a$ and hence $Ker(Res)^F\neq 0$, the contradiction, and therefore $N=R$.

\begin{definition}\label{d7}
Let us say that we are given a simplicial presentation of a prounipotent group $G\cong F(\widehat{X},*)/R$ if
there exists a punctured profinite space $(\widehat{Y}^*,*)$ associated with the chosen bases $Y$ of $R/[R,F](k)^{\vee}$ such that $G$ is included into the following simplicial diagram of prounipotent groups
\begin{equation}
\xymatrix{
{F(\widehat{X}^*\cup \widehat{Y}^*,*)} \ar@<0ex>[r]^{d_0}\ar@<-2ex>[r]^{d_1} & F(\widehat{X}^*,*) \ar@<-2ex>[l]_{s_0} \ar[r]^{\pi} & G}\label{5}
\end{equation}
and for $x \in (\widehat{X}^*,*)$, $y \in (\widehat{Y}^*,*)$ the following identities hold $d_0(x)=x,  d_0(y)=1, d_1(x)=x,  d_1(y)=\tau(y), s_0(x)=x$,
where $\tau$ is a homeomorphism \eqref{eq9} of pointed profinite spaces, identifying $\widehat{Y}^*$ with the pointed profinite space of normal generators in $R(k)$.
\end{definition}

\begin{remark}\label{r7}
As in Section \ref{r1.2}, we can consider a simplicial prounipotent presentation as the second step in the building of the Eilenberg-MacLane complex, we denote it by $F^{(1)}_{\bullet}$, and define the second homotopy group of a pro-unipotent presentation as $\pi_1(F^{(1)}_{\bullet})$
\end{remark}

\begin{proposition}\label{p6}
Let we are given $(X|R) $ a prounipotent presentation \eqref{eq1} of a pro-unipotent group $G$. Then there is a pointed profinite space $(\widehat{Y}^*,*)$ and a convergent map $\tau:(\widehat{Y}^*,*)\rightarrow R(k), \tau(*)=1$, which identify $(\widehat{Y}^*,*)$ with the pointed profinite subspace of normal generators of $R(k)$.

If we assume $d_0(x)=x,  d_0(y)=1, d_1(x)=x, s_0(x)=x,d_1(y)=\tau(y), y\in (\widehat{Y}^*,*)$, then \eqref{5} is a simplicial presentation of $G$, in particular $G\cong F(\widehat{X}^*,*)/d_1(Ker\; d_0)$.
\end{proposition}
\begin{proof}

Let $\tau:(\widehat{Y}^*,*)\rightarrow R(k)\vartriangleleft F(\widehat{X}^*,*)(k)$ is obtained as a restriction to $(\widehat{Y}^*)$ of some lifting of the basis $(\widehat{T}^*,0)$ of $R/[R,R](k)$ along the abelianization map $\mu:R(k)\rightarrow R/[R,R](k)$, it is convergent as mentioned in Remark \ref{r4}.

Define $d_0(x)=x, d_1(x)=x, s_0(x)=x$ for $x\in \widehat{X}^*$ and $d_0(y)=1, d_1(y)=\tau(y)$ for $y\in \widehat{Y}^*$. The existence of such $d_i$ follows from Lemma \ref{l7}, since $F(\widehat{X}^*\cup \widehat{Y}^*,*)$ is the direct sum of $F(\widehat{X}^*,*)$ and $F(\widehat{Y}^*,*)$.
\end{proof}



\subsection{Subpresentations}\label{s2.2}
We now introduce a notion of a subpresentation of a prounipotent simplicial presentation \eqref{5}. To do this, we start with a basis $Y$ which is obtained from the isomorphism of $k$-vector spaces $(R/[R,F](k))^{\vee}\cong k^Y$. A subpresentation is constructed from a subset $Y_1\subset Y$, for this we consider the corresponding embedding $(\widehat{Y_1}^*,*)\hookrightarrow (\widehat{Y}^*,*)$ and hence we define
$$\tau_1=\tau|_{(\widehat{Y_1}^*,*)}:(\widehat{Y_1}^*,*)\rightarrow F(\widehat{X}^*,*)$$
as the restriction of $\tau$ in Definition \ref{d7}. We define $R_0(k)$ as the Zariski normal closure of the profinite subspace $\tau_1(\widehat{Y}_1^*)$ in $F(k)$.
Consider the following diagram:
\begin{equation}\label{7}
\xymatrix{
    R/[R,R](k)\cong k^{(\widehat{Z}^*,*)}  \ar[d]^{\psi} & R_0/[R_0,R_0](k)\cong k^{(\widehat{T_1}^*,*)} \ar[l]_{\phi} \ar[d]^{\psi_0}\\
R/[R,F](k)\cong k^{(\widehat{Y}^*,*)}\ar@<-1ex>[d]_{\vee} \ar@<0ex>[r]^{\xi}  & R_0/[R_0,F](k)\cong k^{(\widehat{Y}_1^*,*)}\ar@<-2ex>[l]_{\chi} \ar@<-1ex>[d]_{\vee}\\
(R/[R,F](k))^{\vee} \cong k^Y \ar@<0ex>[r]^{\chi^{\vee}} \ar@<-1ex>[u]_{*} & k^{Y_1} \ar@<-2ex>[l]_{\xi^{\vee}}\ar@<-1ex>[u]_{*} }
\end{equation}
where $\psi$ and $\psi_0$ are induced by inclusions of the corresponding subgroups $[R,R]\subset[R,F]$ and $[R_0,R_0]\subset[R_0,F]$. $\phi$ is induced by inclusion $R_0\hookrightarrow R$ and $\chi$ is induced from $\phi$ by inclusion $[R_0,F]\subset[R,F]$ and as is shown in Proposition \ref{sub} $\xi$ is actually a projection on a closed subspace.

Taking $d'_0, d'_1$ as restrictions of $d_0, d_1$ to $(\widehat{Y_1}^*,*)$ in \eqref{5} we get the diagram, which is called \textbf{the subpresentation of the simplicial presentation} \eqref{5}.
\begin{equation}
\xymatrix{ F(\widehat{X}^*\cup \widehat{Y_1}^*,*) \ar@<0ex>[r]\ar@<-1ex>[r]_(0.6){d'_0, d'_1} & F(\widehat{X}^*,*) \ar@<-1ex>[l]_(0.4){s_0} \ar[r] & G_0}\label{6}.
\end{equation}

\begin{proposition}\label{sub}
Let we a given a prounipotent presentation \eqref{5}, then \eqref{6} is a simplicial presentation.
\end{proposition}
\begin{proof}
We only need to check that $(R_0/[R_0,F](k))^{\vee}\cong k^{Y_1}$.
Let $\Upsilon$ be the closed subspace of $R_0/[R_0,R_0](k)$ generated by the image of $\tau_1(Y_1^*)\subset R_0(k)$ under the projection $R_0(k)\rightarrow R_0/[R_0,R_0](k)$. Since $\tau_1(Y_1^*)$ generates $R_0(k)$ as a Zariski normal subgroup of $F(k)$, we see that $\psi_0(\Upsilon)$ generates $R_0/[R_0,F](k)$ as a linearly-compact vector space and, therefore, $\psi_0:\Upsilon\rightarrow R_0/[R_0,F](k)$ is an epimorphism. By construction, $\phi(\Upsilon)\cong k^{(\widehat{Y}_1^*,*)}$ and, therefore, $\Upsilon\cong k^{(\widehat{Y}_1^*,*)}$, so we only need to show that $\psi_0:\Upsilon \rightarrow R_0/[R_0,F](k)$ is a monomorphism. Consider the composition $\chi\circ\psi_0: \Upsilon\rightarrow R/[R,F](k)\cong k^{(\widehat{Y}_1^*,*)}\oplus k^{(\widehat{Y\setminus Y_1}^*,*)}$, this is a homeomorphism of linearly-compact vector spaces, mapping the basis $(\widehat{Y}_1^*,*)$ of
$\Upsilon\cong k^{(\widehat{Y}_1^*,*)}$ to the basis $(\widehat{Y}_1^*,*)$ of the closed subspace $k^{(\widehat{Y}_1^*,*)}\subset R/[R,F](k)$ and hence $\chi\circ\psi_0: \Upsilon\rightarrow R/[R,F](k)$ is a monomorphism and therefore $\psi_0$ is also mono, so, by duality, $(R_0/[R_0,F](k))^{\vee}\cong k^{Y_1}$.
\end{proof}

\subsection{Prounipotent (pre-)crossed modules}\label{s3.3}
Pro-$p$-crossed modules were extensively studied by T. Porter in \cite{Por} and pro-unipotent crossed modules were introduced and studied in the case of finite type presentations in \cite{Mikh2016}. The purpose of this section is to extend the previous concepts to prounipotent crossed modules over punctured profinite spaces. We do not divide the presentation into zero characteristics and the pro-$p$-case, since they are completely analogous.

A right action of a prounipotent group $G_1$ on a prounipotent group $G_2$ is a natural transformation of the group-valued functors $G_2\times G_1\rightarrow G_2$ \cite[2.6]{Jan}. It follows from Yoneda's Lemma that $\mathcal{O}(G_2)$ inherits the left $\mathcal{O}(G_1)$-comodule structure.
We always assume, that such an action is \textbf{prounipotent}, i.e. there is a system of $G_1$-invariant normal subgroups of finite codimension $\{N_{\lambda}\}$ of $G_2$ and $G_2\cong \varprojlim G_2/N_{\lambda}$, where $G_2/N_{\lambda}$ are unipotent groups.
Our definition is equivalent to the definition of the continuous action \cite[1.3]{Por} in the pro-$p$-case and
it follows, that in the pro-$p$-case the corresponding semidirect product $G_2\leftthreetimes G_1$ is a pro-$p$-group.
\begin{lemma}\label{l6}
Let we are given a prounipotent action of a prounipotent group $G_1$ on a prounipotent group $G_2$, then the corresponding semidirect product $G_2\leftthreetimes G_1$ is the prounipotent group.
\end{lemma}
\begin{proof} Now suppose $G_2$ is unipotent.
Let $\mu:\mathcal{O}(G_2)\rightarrow\mathcal{O}(G_1)\otimes\mathcal{O}(G_2)$ defines the left $\mathcal{O}(G_1)$-comodule structure on $\mathcal{O}(G_2)$. Since $\mu(\mathcal{O}(G_2))$ is finitely generated as an algebra and any finite set in a Hopf algebra is contained in a finitely generated Hopf subalgebra \cite[Theorem 3.3]{Wat}, we conclude that $\mu(\mathcal{O}(G_1))\subseteq A\otimes \mathcal{O}(G_2)$, where $A$ is a finitely generated (as algebra) Hopf subalgebra in $\mathcal{O}(G_1)$. Since any Hopf algebra is faithfully flat over its Hopf subalgebra, it follows from \cite[Theorem 13.2]{Wat} that we have an epimorphism $Spec(\mathcal{O}(G_1))\rightarrow Spec(A)$ and hence the action is factored via the unipotent group scheme $Spec(A)$ and, therefore,
$G_1\leftthreetimes G_2\cong \varprojlim_{W_{\beta} \supseteq A} G_1\leftthreetimes (G_2 / W_{\beta})$. It remains to note that the semidirect product of unipotent groups is a unipotent group (since it contains a finitely generated, nilpotent discrete torsion-free dense subgroup Remark \ref{r6}) and hence $G_1\leftthreetimes G_2$ is a prounipotent group.

In the general case, when $ G_2 $ is a prounipotent group with a unipotent action of $G_1$, we have the decomposition
 $G_1\leftthreetimes G_2\cong \varprojlim G_1/N_{\lambda}\leftthreetimes G_2$, where $G_1/N_{\lambda}$ are unipotent, and therefore our statement follows from the previous case.
\end{proof}
\begin{example}
It is worth to note that the action of a prounipotent group on itself by conjugation is always prounipotent. It is also shown in the proof of Proposition \ref{p2} below that the natural action of a prounipotent group on its relation module is prounipotent.
\end{example}
\begin{definition} \label{d1}
By a prounipotent pre-crossed module one calls $(G_2,G_1,\partial)$ a triple, where $G_1,G_2$
are prounipotent groups, $\partial: G_2 \to G_1$ is a homomorphism of prounipotent groups, $G_1$ acts prounipotently on $G_2$
from the right, satisfying, for any $k$-algebra $A$, the identity
$$\text{CM 1) } \partial(g_2^{g_1}) = g_1^{-1} \partial(g_2) g_1,$$
where the action is written in the form $(g_2,g_1) \to g_2^{g_1}$, $g_2 \in G_2(A), g_1 \in G_1(A).$
\end{definition}

\begin{definition}
A prounipotent pre-crossed module is called crossed if  for any $k$-algebra $A$ the following additional identity holds for $g_1,g_2\in G_2(A)$
$$\text{CM 2) }g^{\partial(g_1)}= g_1^{-1} g g_1.$$
The identity CM 2) is called the Peiffer identity.
Such prounipotent crossed module will be denoted by $(G_2,G_1,\partial)$.
\end{definition}

\begin{definition}\label{d12}
A homomorphism $(G_2,G_1, \partial) \to (G'_2,G'_1, \partial')$ of prounipotent crossed modules
is a pair of homomorphisms
$\varphi: G_2 \to G'_2 \text{ and } \psi: G_1 \to G'_1$ of prounipotent groups satisfying:
a). $\varphi(g_2^{g_1}) =\varphi (g_2)^{\psi(g_1)}$; and b). $\partial' \varphi(g_2) = \psi \partial(g_2).$
\end{definition}

Below we will define the notion of a free (pre-)crossed module and show that the prounipotent (pre-)crossed module of \eqref{5} is free.

\begin{definition} \label{d2}
A prounipotent (pre-)crossed module $(G_2,G_1,d)$ is called a free prounipotent (pre-)crossed module on a pointed profinite space
$(\widehat{Y}^*,1)\in G_2(k)$ if $(G_2,G_1,d)$
possesses the following universal property with respect to convergent maps $\nu: (\widehat{Y}^*,1)\rightarrow G'_2(k)$, i.e.
for any prounipotent (pre)crossed module $(G'_2,G'_1,d')$ and a convergent map $\nu: (\widehat{Y}^*,1)\rightarrow G'_2(k)$
and for any homomorphism of prounipotent groups $f:G_1\rightarrow G'_1$ such that
$fd(\widehat{Y}^*,1)=d'(k)\nu(\widehat{Y}^*,1),$ there is a unique homomorphism
of prounipotent groups $h:G_2\rightarrow G'_2$ such that $h(k)(\widehat{Y}^*,1)=\nu(\widehat{Y}^*,1)$
and the pair $(h,f)$ is a homomorphism of (pre-) crossed modules.
\end{definition}

\begin{proposition} \label{p1}
Assume we are given a simplicial prounipotent presentation \eqref{5}, then
$Ker\; d_0 \xrightarrow{d_1} F(\widehat{X}^*,*)$ is a free prounipotent pre-crossed module on $(\widehat{Y}^*,1)\subset Ker\; d_0(k)\subset F(\widehat{X}^*\cup \widehat{Y}^*,*)(k)$.
\end{proposition}
\begin{proof} Let $A\xrightarrow{\partial} H$ be some prounipotent pre-crossed module.
For each convergent map $\nu:(\widehat{Y}^*,1) \rightarrow A(k)$
we must construct a homomorphism $\psi$ in the diagram as follows

$$
\xymatrix{
{Ker\; d_0} \ar[r] \ar@/^/[dr]^{\varphi} \ar@/^1pc/[rr]^{d_1} & F(\widehat{X}^*\cup \widehat{Y}^*,*) )  \ar[r]^{d_0} \ar[d]^{\psi} & F(\widehat{X}^*,*) \ar[d]^{f}
\\
Y \ar[r]^{\nu} \ar[ru] & A\leftthreetimes H \ar[r]^{\widetilde{\partial}} & H}
$$
where $\widetilde{\partial}(A)=\partial(A)$, $\widetilde{\partial}(H)=id_H$, $\widetilde{\partial}(a,b)=\partial(a)\cdot b$
for $a\in A(k), b\in H(k)$.
The semidirect product $A \leftthreetimes H$ is a prounipotent group,
since the action in the pre-crossed module $(A,H,\partial)$ is prounipotent.
The fact that $\widetilde{\partial}$ is a homomorphism follows by direct computation using the fact that
for $\partial$ the property CM 1) from Definition \ref{d1} holds. Indeed,
$\widetilde{\partial}((a,b)(a_1,b_1))=\widetilde{\partial}(aa_1,b^{a_1}b_1)=a a_1\cdot\partial( b^{a_1}b_1)=aa_1\partial(b^{a_1})\partial(b_1)=aa_1a_1^{-1} \partial(b)a_1\partial(b_1)=a\partial(b)a_1\partial(b_1).$

We define $\psi:F(\widehat{X}^*\cup \widehat{Y}^*,*)\rightarrow A\leftthreetimes H$ on $k$-points by the rule $\psi(k)(x)=f(k)(x)$ on $x\in \widehat{X}^*$ and $\psi(k)(y)=\nu(y)$ on $y\in \widehat{Y}^*.$
Lemma \ref{l7} yields a homomorphism of prounipotent groups
$\psi: F(\widehat{X}^*\cup \widehat{Y}^*,*)\rightarrow A\leftthreetimes H.$

Put $\varphi: Ker\;d_0\rightarrow A\leftthreetimes H$
equals to $\varphi=\psi\mid_{Ker\; d_0}$, i.~e. restricting $\psi$ to $Ker\; d_0$.
Then the pair $(\phi,f)$ is a pre-crossed module homomorphism. Indeed, let $g_1\in F(\widehat{X}^*,*)$, then there is $\widetilde{g}_1\in F(\widehat{X}^*\cup \widehat{Y}^*,*)$ such that $d_1(\widetilde{g}_1)=g_1$, and hence we have
$\phi(g_2^{g_1})=\phi(\widetilde{g}_1^{-1}g_2\widetilde{g}_1)=\phi(\widetilde{g}_1)^{-1}\phi(g_2)\phi(\widetilde{g}_1)$ and, as $A\leftthreetimes H$ is a semidirect product and $\psi|_{F(\widehat{X}^*,*)}=f$, $\phi(\widetilde{g}_1)^{-1}\phi(g_2)\phi(\widetilde{g}_1)=\phi(g_2)^{\phi(\widetilde{g}_1)}=\phi(g_2)^{f(g_1)}$ and a). of Definition \ref{d12} holds and it is enough and easy to check b). of Definition \ref{d12} for $\widehat{X}^*\cup \widehat{Y}^*$.
\end{proof}

We use the ideas of Section \ref{s1.0} to construct the prounipotent crossed module of the prounipotent presentation \eqref{5}.
Let $P=<a^{d_1(u)}=u^{-1}au>$, where $u,a\in Ker\; d_0(k)$,
the normal Zariski closure of the subgroup of Peiffer commutators and let $C=Ker\; d_0/P$. We say that
$C \xrightarrow{\overline{d_1}} F,$ where $\overline{d_1}$ is induced from $d_1$ and $F$ acts on $Ker\; d_0$ by conjugations $a^u=s_0(u)^{-1}as_0(u)$, is the prounipotent crossed module of the prounipotent presentations \eqref{5}.

The Brown-Loday Lemma (see Section \ref{r1.2}) identifies the discrete commutant $[Ker\; d_0,Ker\; d_1]$ with the image $d^2_2(NG_2)$ (which is always closed) and hence abstract commutant $[Ker\; d_0,Ker\; d_1]$ is Zariski closed.
As proven in Lemma \ref{l3.3} $[Ker\; d_0,Ker\; d_1]=P$ and therefore by the crossed module of a prounipotent presentation \eqref{5} we also mean the
crossed module from Lemma \ref{l3.3} (also \cite[Lemma 3.10]{Mikh2016}).

For any prounipotent presentation \eqref{5}, in analogy with Proposition \ref{l3}, we introduce its second homotopy group, as already suggested in Remark \ref{r7}, as
$u_2(X|R)=\pi_1(F^{(1)}_{\bullet}),$
which is the same as
$$u_2(X|R)=Ker\;( C(k)\xrightarrow{\overline{d}_1} F(k)).$$
If for a given prounipotent presentation \eqref{5} we have $u_2(X|R)=0$, then we will call such presentation \textbf{aspherical}, and we are ready to present the main result of the paper.
\begin{theorem}\label{t4}
Let \eqref{5} be an aspherical presentation of a prounipotent group $ G \cong F (X) / R $ over the base field $ k $ of characteristics zero (or $ p \geq 2 $ is a prime and \eqref{5} is a pro-$p$-presentation of a pro-$p$-group $ G $) and let \eqref{6} be a subpresentation of \eqref{5}, then \eqref{6} is also aspherical.
\end{theorem}


\section{Relation modules and aspherical presentations}\label{s3}
By relation module of a pro-$p$-presentation \eqref{eq1} (see also \ref{s2.1}) of a pro-$p$-group $G\cong F/R$ we understand the abelianization $\overline{R}=R/[R,R]$, where $[R,R]$ is the commutant (closed in the pro-$p$-topology), with the conjugation-induced action of $G$ on $\overline{R}$.
Relation modules of pro-$p$-groups are known to be $\mathbb{F}_pG$-modules \cite[5.7]{ZR} and therefore they are topological $\mathcal{O}(F)^*$-modules in our terminology. Similar structures have been discovered in \cite{Mag} and \cite{Mikh2016} for prounipotent presentations with finiteness assumptions.
In Section \ref{s3.1} we prove that relation modules without any assumptions of finiteness are naturally topological modules with the conjugation-induced action.

The prounipotent Gasch\"{u}tz Lemma and the reduction of the study of non-proper presentations to the theory of proper presentation are contained in Section \ref{s4.1}.

In Section \ref{s3.2} we show that the presentation \eqref{5} is aspherical if and only if its relation module is free, we also establish a cohomological criterion of asphericity, namely it is $cd\; G\leq 2$.
Similar results for pro-$p$-groups were already known in \cite[Proposition 7.7]{Koch} and for proper finite type prounipotent presentations over fields of zero characteristics in \cite{LM2} and \cite{Mikh2017}.

In Section \ref{s1.4} we adapt V.M.Tsvetkov's ideas \cite{Ts} about pro-$p$-groups of cohomological dimension 2, obtaining a criterium for asphericity for subpresentations, we also replace his dubious argument in the proof of equivalence (a), (b) and (d), (e) on the spectral sequence argument prepared in Section \ref{s2.3}.
Section \ref{s3.5} contains the proof of Theorem \ref{t4}.
\subsection{Relation modules of prounipotent presentations}\label{s3.1}

\begin{proposition} \label{p2}
Suppose we are given $S\lhd F=F(\widehat{X}^*,*)$ a normal subgroup of a free prounipotent group $F$ over a pointed profinite space $(\widehat{X}^*,*)$. Then the rational points $\overline{S}(k)$ of the abelianization $\overline{S}=S/[S,S]$ can be naturally endowed with the structure of a topological $\mathcal{O}(F)^*$-module with the conjugation induced action of $F$ on $\overline{S}$.
\end{proposition}

\begin{proof}
First, we consider the case when $X$ is finite, then $S(k)$ has a convergent basis \cite{Mag}, say $Z\in S(k)$, that is $Z$ generates $S(k)$ as a Zariski-normal subgroup and all but finitely many elements of $Z$ do lie in the lower series filtration $L_n\leq F=F(X)(k)$. In fact, the quotients $L_i/L_{i+1}\cong C_i/C_{i+1}\otimes k$ \cite[Corollary 3.7, A.3]{Qui5} (where $C_i$ is a lower central filtration of the free discrete group $\Phi(X)$)
have finite dimension \cite[Theorem 1, Chapter I,4.6]{Se}, hence the
quotients $S(k)\cap L_i(k)/S(k)\cap L_{i+1}(k)$ are also of finite dimension and, therefore, taking their lifts in $S(k)$ we get a convergent basis.
As a consequence, we see that the quotients  $\overline{S}_i(k)=\overline{S(k)/S(k)\cap L_i(k)}$ are in fact finite-dimensional vector spaces over $k$.

We want to show that each $\overline{S}_i(k)$ with the conjugation-induced action of $(F/L_i)(k)$ (we take into account that conjugation action of $L_i(k)$ on $\overline{S}_i(k)$ is trivial) is a topological $\mathcal{O}(F/L_i)^*$-module.
By Definition \ref{d10}, this is equivalent to finding a continuous $k$-linear mapping $$\mu_i:\overline{S}_i(k)\widehat{\otimes} \mathcal{O}(F/L_i)^*\rightarrow \overline{S}_i(k),$$
arising naturally from the conjugation action and included into the necessary diagrams.

By \cite[Theorem A.2., Corollary A.4.]{HM2003} $\Phi(X)/C_i\subset F(k)/L_i(k)$ and hence $\Phi(X)/C_i$ also acts on $S_i(k)=S(k)/S(k)\cap L_i(k)$ by conjugations. We continue this action by linearity to the group ring $k[\Phi(X)/C_i]$:
$$\widetilde{\mu}_i:\overline{S}_i(k)\otimes k[\Phi(X)/C_i]\rightarrow \overline{S}_i(k),$$ given by the rule
$r\otimes\sum a_jf_j\mapsto \sum a_jr^{f_j}, a_j\in k, f_j\in \Phi(X)/C_i,r\in \overline{S}_i(k)$.

The required mapping $\mu_i$ will be obtained as a completion of this map. We start from Quillen's formulae $\mathcal{O}(F/L_i)^*\cong \varprojlim_n k[\Phi(X)/C_i]/I_{\Phi}^n$ and as
$\Phi(X)/C_i$ is torsion-free nilpotent, then $\cap I_{\Phi}^n=0$, where $I_{\Phi}$ is the augmentation ideal of $k[\Phi(X)/C_i]$,
and hence we get the decomposition $\overline{S}_i(k)\cong \varprojlim_n \overline{S}_i(k)/\overline{S}_i(k)I^n_{\Phi},$
where $\overline{S}_i(k)I^n_{\Phi}$ denotes the $k$-linear subspace of $\overline{S}_i(k)$ generated by elements of the form $a\cdot f,$ where $a\in \overline{S}_i(k)$, $f\in I^n_{\Phi}$.

$F$ acts prounipotently on $\overline{S}$. Indeed, $F$ acts trivially on graded quotients $S\cap L_i/S\cap L_{i+1}$, and we take their images under abelianization of $S/S\cap L_i$ to obtain the quotients $\overline{S/S\cap L_i}$ of the prounipotent filtration of $\overline{S}$.

Let $\widetilde{G}=\overline{S}_i(k)\leftthreetimes F(X)/L_i$, by Lemma \ref{l6} $\widetilde{G}$ is a unipotent group and
we have well defined, by Quillen's formulae, closed subspaces $\overline{S}_i(k)\cdot\widehat{I}^m\subseteq\overline{S}_i(k)$, where
$\widehat{I}$ is the augmentation ideal of $\mathcal{O}(F/L_i)^*$. As $\overline{S}_i(k)$ is of finite dimension,
 then there is $m\in \mathbb{N}$ such that $\overline{S}_i(k)\cdot\mathfrak{I}^m=0$, where $\mathfrak{I}\subset \mathcal{O}(\widetilde{G})^*$ is the augmentation ideal, but
$\cap\overline{S}_i(k)I^n_{\Phi}\subseteq\cap\overline{S}_i(k)\widehat{I}^n\subseteq\cap\overline{S}_i(k)\mathfrak{I}^n=0$ and therefore $\overline{S}_i(k)I^n_{\Phi}=\overline{S}_i(k)\widehat{I}^n$ for each $n\in \mathbb{N}$.
$(\overline{\widetilde{S}}_i\leftthreetimes \Phi(X)/C_i)^{\wedge}_u(k)\cong \overline{S}_i(k)\leftthreetimes F/L_i(k)$, by Remark \ref{r6}, where $\overline{\widetilde{S}}(k)$ is taken finitely generated and dense in $\overline{S}_i(k)$, now we see that the conjugation action of $F/L_i(k)$ on $\overline{S}_i(k)$ coincides with the action defined by completion, and therefore the last one is independent of the chosen basis $X$ in $F/L_i(k)$.
Since $\overline{S}_i(k)/\overline{S}_i(k)I^n_{\Phi}\cong \overline{S}_i(k)/\overline{S}_i(k)\widehat{I}^n$ and $\mathcal{O}(F)^*/\widehat{I}^n\cong k[\Phi(X)/C_i]/I^n_{\Phi}$, for each $n\in \mathbb{N}$ we have the $k$-linear maps of finite dimensional vector spaces
$$\overline{S}_i(k)/\overline{S}_i(k)\widehat{I}^n\otimes \mathcal{O}(F/L_i)^*/\widehat{I}^n\rightarrow \overline{S}_i(k)/\overline{S}_i(k)\widehat{I}^n,$$
which are completely defined by the conjugation action of $F/L_i,$ since by \cite[Proposition 3.6 (a)]{Qui5} for $n>i$ the canonical map $\Phi(X)/C_i\rightarrow k[\Phi(X)/C_i]/I^n_{\Phi}$ is injective and $\Phi(X)/C_i$ defines this action.
Passing to the inverse limit over $n$, we obtain the required $k$-linear mapping $\mu_i$, which fits into the required diagrams, since the mappings $\widetilde{\mu}_i$ are taken from the associative group ring action with unit and, therefore, $\overline{S}_i(k)$ is the topological $\mathcal{O}(F/L_i)^*$-module.
Passing to the inverse limit over $i$ we get $\mathcal{O}(F)^*$-module structure on $\overline{S}(k)$.

Let $X$ be a chosen basis in $(F/[F,F])^{\vee}$, then we present $F=F(X)(k)=F(\widehat{X},*)(k)$ as the inverse limit $F\cong\varprojlim_{\lambda} F(X_{\lambda},*)$ over finite pointed spaces $(X_{\lambda}^*,*)$ as in \eqref{eq3}.
As a result, we obtain the decomposition
$S(k)\cong\varprojlim S_{\lambda}(k)$ into the inverse limit of
$S_{\lambda}(k)=\phi_{\lambda}(S)(k),$ where $\phi_{\lambda}:F\rightarrow F(X_{\lambda}^*,*)$
- the natural projections and, therefore, the
decomposition $\overline{S}=\varprojlim \overline{S}_{\lambda}$. For each $\lambda$ we have $\mathcal{O}(F(X_{\lambda}^*,*))^*$-module structure $\mu_{\lambda}:\overline{S}_{\lambda}(k)\widehat{\otimes} \mathcal{O}(F_{\lambda})^*\rightarrow \overline{S}_{\lambda}(k)$, induced by the conjugation action of $F_{\lambda}=F(X_{\lambda}^*,*)(k)$ on $S_{\lambda}(k)$
and, as the conjugation action commute with factorization, we can take $\mu=\varprojlim \mu_{\lambda}$ in order to obtain the structure of $\mathcal{O}(F)^*$-topological module on $\overline{S}(k)$. By construction, taking in mind Proposition \ref{p001}, this action is also independent of the chosen basis $X$.
\end{proof}

As a byproduct of ideas in the proof, we obtain the following useful corollaries

\begin{corollary}\label{c3}
Let $M$ and $N$ be topological $\mathcal{O}(G)^*$-modules and let $f:M\rightarrow N$ be a $G(k)$-homomorphism of linearly-compact vector spaces, then $f$ is a homomorphism of topological $\mathcal{O}(G)^*$-modules.
\end{corollary}

\begin{corollary}\label{c1}
Let \eqref{eq1} be a prounipotent presentation, 
then the linearly compact $k$-vector space $\overline{R}(k)=R/[R,R](k)$ can be naturally endowed with the structure of a topological $\mathcal{O}(G)^*$-module with the action induced by the conjugations of $F$ on $R$.
We call such a topological $\mathcal{O}(G)^*$-module the relation module of a prounipotent presentation.
\end{corollary}
\begin{proof} $R\lhd F$ lies in the kernel of the conjugation action of $F$ on $\overline{R}$ and hence the action factors through the action of $G$ and, therefore, $\overline{R}(k)$ is topological $\mathcal{O}(G)^*$-module.
\end{proof}

\begin{proposition}\label{p8}
Suppose we are given a prounipotent presentation in the simplicial form \eqref{5}. Then $\overline{C}(k)$ inherits
structure of a topological $\mathcal{O}(G)^*$-module.
\end{proposition}
\begin{proof}
By Proposition \ref{p2}, $\overline{Ker\; d_0}(k)$ - the abelianization of $Ker\; d_0$ is naturally a $\mathcal{O}(F)^*$-module. It contains $\langle[Ker\;d_0,Ker\; d_1](k)\rangle$ - the $\mathcal{O}(F)^*$-submodule of $\overline{Ker\; d_0}(k)$ generated by the image of $[Ker\;d_0,Ker\; d_1](k)$ in $\overline{Ker\; d_0}(k)$.

Since $\overline{C}(k)\cong \overline{Ker\; d_0}(k)/\langle[Ker\;d_0,Ker\; d_1](k)\rangle$ we get that $\overline{C}(k)$ is a topological $\mathcal{O}(F)^*$-module.
The rest is to show, that $R(k)$ lies in the kernel of this action.
Lets take $d_1(x)\in R(k)=d_1(Ker\;d_0(k))$, $x,y\in Ker\;d_0(k)$ by \text{CM 2)} of Definition \ref{d1.2} we have $x^{-1}yx=y^{d_1(x)}$ and hence $x^{-1}yxy^{-1}=y^{d_1(x)}y^{-1}$ - the image of $y^{d_1(x)}y^{-1}$ in $\overline{C}(k)$ is zero.
\end{proof}

\begin{definition}\label{d6}
Let $G$ be a prounipotent group, we say that the mapping $\mu:(\widehat{X}^*,*)\rightarrow L$ of a pointed profinite space $(\widehat{X}^*,*)$ into a topological $\mathcal{O}(G)^*$-module $L$ is convergent if there exists a decomposition $L\cong \varprojlim L_{\lambda}$ into the inverse limit of finite dimensional $\mathcal{O}(G)^*$-modules $L_{\lambda}$, such that for all but finite $x\in \widehat{X}^*$ we have $\phi_{\alpha}\mu(x)=1$, where $\phi_{\lambda}:L\rightarrow L_{\lambda}$ is the projection.
We also call the topological $\mathcal{O}(G)^*$-module $M$ free over the pointed profinite space $i:(\widehat{X}^*,*)\rightarrow M$, $i(*)=0$, we denote it $(\mathcal{O}(G)^*)^{(\widehat{X}^*,*)}$, if it has the following universal property:
for any toplogical $\mathcal{O}(G)^*$-module $L$ and a convergent mapping $\mu:(\widehat{X}^*,*)\rightarrow L$ there exists a homomorphism $\widehat{\mu}: M\rightarrow L$ of topological $\mathcal{O}(G)^*$-modules such that $\mu=\widehat{\mu}\circ i$.
\end{definition}

\begin{lemma} For any pointed profinite space $(\widehat{X}^*,*)$ there is a free topological $\mathcal{O}(G)^*$-module $M=(\mathcal{O}(G)^*)^{(\widehat{X}^*,*)}$ over $(\widehat{X}^*,*)$.
\end{lemma}
\begin{proof}
Let $(\widehat{X}^*,*)\cong \varprojlim (\widehat{X_{\lambda}}^*,*)$, it is easy to check that the toplogical $\mathcal{O}(G)^*$-module $\varprojlim (\mathcal{O}(G)^*)^{(\widehat{X_{\lambda}}^*,*)}$ possesses the required universal properties, its dual $M^{\vee}\cong \mathcal{O}(G)^X$ is an induced $G$-module.
\end{proof}

\begin{remark}
By Proposition \ref{t3} (below) $H^1(R,k_a)$ is endowed with a structure of $\mathcal{O}(G)$-comodule, and $H^1(R,k_a)$ can be identified with $\overline{R}(k)^{\vee}$ \cite[Proposition 10]{Mikh2017} and, therefore, by duality, $\overline{R}(k)$ is a topological $\mathcal{O}(G)^*$-module. Such $\mathcal{O}(G)^*$-module structure is the same as $k\langle\langle G\rangle\rangle =End_G(\mathcal{O}(G)),\mathcal{O}(G))$-module structure on relation modules of \cite{Mag}, \cite{LM2}, since by \eqref{eq10} $k\langle\langle G\rangle\rangle \cong \mathcal{O}(G)^*$ and by \cite[Lemma 8(1)]{Mag} the action of $G(k)$ on $\overline{R}(k)$ is induced by conjugations.

It is also proved in Proposition \ref{t3} that $H^1(R,k_a)$ is naturally embedded into its injective hull, which is $\mathcal{O}(G)^Y$, and, therefore, by duality as well, we have a natural epimorphism of $\mathcal{O}(G)^*$-modules $(\mathcal{O}(G)^*)^{(\widehat{Y}^*,*)}\rightarrow \overline{R}(k)$.
\end{remark}
\begin{remark}\label{r10}
Let $C(k)\xrightarrow{\overline{d}_1} F(k)$ be the free prounipotent crossed module of Section \ref{s3.3}. Then $\overline{d}_1$ induces a homomorphism, say $\partial:\overline{C}(k)=(\mathcal{O}(G)^*)^{(\widehat{Y}^*,*)}\rightarrow \overline{R}(k)$, of topological $\mathcal{O}(G)^*$-modules, now $\partial_G:\overline{C}(k)_G=k^{(\widehat{Y}^*,*)}\rightarrow \overline{R}(k)_G=k^{(\widehat{Y}^*,*)}$ - the induced homomorphism on coinvariants is the ``fixed'' isomorphism of topological vector spaces of Remark \ref{r9}.
\end{remark}
\subsection{Non-proper presentations and their relation modules}\label{s4.1}
The following prouniptent analog of the Gasch\"{u}tz Lemma,
which was originally proved by Gasch\"{u}tz for finite groups and later
extended by Fried-Jarden to pro-finite groups \cite[Lemma 17.7.2]{FJ}, enlightens the roots of the Gasch\"{u}tz result. We do not need this result in the rest of the paper, however its analog is crucial in the study of non-proper pro-finite presentations \cite{Lub}.
\begin{proposition}\label{l8}
Let $f: G \rightarrow H$ be an epimorphism of prounipotent
groups with $rank(G)\leq e$. Let $h_1, . . . , h_e$ be a system of generators of $H$.
Then there exists a system of generators $g_1, . . . , g_e$ of $G$ such that $f(g_i) = h_i$,
$i = 1, . . . , e$.
\end{proposition}
\begin{proof}
As we have mentioned this statement is already proven for profinite groups and in particular it is true for pro-$p$-groups, therefore, we need to consider the case, when $char(k)=0$.
Let, for first,  $f: G \rightarrow H$ be a homomorphism of unipotent groups.
By the functorial correspondence Example \ref{e1.1} between unipotent groups over a field $k$ of characteristics 0 and nilpotent Lie algebras over $k$, this $f$ gives rise to the epimorphism $\mathfrak{f}:\mathfrak{g}\rightarrow \mathfrak{h}$ of the nilpotent Lie algebras. Since $f$ is a Lie algebra homomorphism, then $f([\mathfrak{g},\mathfrak{g}])=[\mathfrak{h},\mathfrak{h}]$ and, therefore, we obtain the following commutative diagram
$$\xymatrix{0 \ar[r] &\mathfrak{u} \ar[r]\ar[d]^{\tau_{\mathfrak{g}}|_{\mathfrak{u}}} & \mathfrak{g} \ar[r]^{f} \ar[d]^{\tau_{\mathfrak{g}}} & \mathfrak{h} \ar[d]^{\tau_{\mathfrak{h}}} \ar[r]&0 \\
0 \ar[r] & \widetilde{\mathfrak{u}}   \ar[r] & \mathfrak{g}/[\mathfrak{g},\mathfrak{g}] \ar[r]^{\widetilde{\mathfrak{f}}}  & \mathfrak{h}/[\mathfrak{h},\mathfrak{h}] \ar[r]&0}.$$
Consider a vector space decomposition
$\mathfrak{g}\cong \mathfrak{h}_1\oplus \mathfrak{u}$, where $\mathfrak{f}|_{\mathfrak{h}_1}: \mathfrak{h}_1\rightarrow \mathfrak{h}$ is the isomorphism of vector spaces,
since $[\mathfrak{u},\mathfrak{u}]\subseteq[\mathfrak{g},\mathfrak{g}]\cap \mathfrak{u}$, it follows, that nilpotent Lie algebra generators of $\mathfrak{u}$ cover a basis of $\widetilde{\mathfrak{u}}$.

Since $\mathfrak{f}$ induces the epimorphism of abelianizations  $\overline{\mathfrak{f}}:\mathfrak{g}/[\mathfrak{g},\mathfrak{g}]\rightarrow \mathfrak{h}/[\mathfrak{h},\mathfrak{h}]$ and since $h_1, . . . , h_e$ generate $H$, then we can choose $\widetilde{h}_1=log(h_1), . . . , \widetilde{h}_l=log(h_l)$, where $l=dim_k(\mathfrak{h}/[\mathfrak{h},\mathfrak{h}])$, - a minimal subset, which generate $\mathfrak{h}$ as a nilpotent Lie algebra, i.e. a minimal subset of $Im\; \tau_{\mathfrak{h}}(\widetilde{h}_i)$ which form a basis in $\mathfrak{h}/[\mathfrak{h},\mathfrak{h}]$.
Let $\widetilde{g}_1,..,\widetilde{g}_l\in \mathfrak{h}_1$, such that $\mathfrak{f}(\widetilde{g}_i)=\widetilde{h}_i, i=1,..,l$, and there are exist $\widetilde{g}'_{l+1},..,\widetilde{g}'_d\in \mathfrak{u}$, where $d=d(\mathfrak{g})=dim_k(\mathfrak{g}/[\mathfrak{g},\mathfrak{g}])$, such that, $\tau_{\mathfrak{g}}|_{\mathfrak{u}}(\widetilde{g}'_i)$ generate $\widetilde{\mathfrak{u}}$.
Now we define $\widetilde{g}_{l+1},..,\widetilde{g}_d\in \mathfrak{g}$ by the rule $\widetilde{g}_i=\widetilde{g}'_i+r_i$, where $r_i\in \mathfrak{h}_1$, such that $\mathfrak{f}(r_i)=\widetilde{h}_i$, $i=l+1,..,d$. For $i>d$ we take arbitrary $\widetilde{g}_i$, such that $\mathfrak{f}(\widetilde{g}_i)=\widetilde{h}_i$.
By construction $\widetilde{g}_i$ generate $\mathfrak{g}$ and $\mathfrak{f}(\widetilde{g}_i)=\widetilde{h}_i$ and hence the result follows for unipotent groups as the elements $g_i=exp(\widetilde{g}_i)\in G(k)$ posses the required properties.

For $\mathfrak{g}$ general pronilpotent Lie algebra we will follow the ideas in the proof of \cite[Lemma 17.7.2]{FJ}. Present $\mathfrak{f}: \mathfrak{g} \rightarrow \mathfrak{h}$ as an inverse limit of epimorphisms
of nilpotent Lie algebras $\mathfrak{f}_i: \mathfrak{g}_i \rightarrow \mathfrak{h}_i$. Specifically, if $j \geq i$, then there are
epimorphisms $\psi_{ji}: \mathfrak{g}_j \rightarrow \mathfrak{g}_i$ and $\phi_{ji}: \mathfrak{h}_j \rightarrow \mathfrak{h}_i$ such that $\mathfrak{f}_i\circ \psi_{ji} = \phi_{ji} \circ \phi_j$.
In addition, there are epimorphisms $\kappa_i: \mathfrak{g} \rightarrow \mathfrak{g}_i$ and $\rho_i: \mathfrak{h} \rightarrow \mathfrak{h}_i$ such that
$\mathfrak{f}_i \circ \kappa_i = \rho_i\circ \mathfrak{f}$.

Let $\textbf{h}\in \mathfrak{h}^e$ be a $e$-tuple, that generate $\mathfrak{h}$ as the pronilpotent Lie algebra. For each $i \in I$ denote the vector subspace generated by $e$-tuples $\textbf{x} \in \mathfrak{g}^e_i$ that generate
$\mathfrak{g}_i$ and satisfy $\mathfrak{f}_i(x) = \rho_i(\textbf{h})$ by $\mathfrak{A}_i$. By the case where $\mathfrak{g}$ is unipotent, $\mathfrak{A}_i$ is
nonempty. In addition, $\mathfrak{A}_i$ is of finite dimension. If $j \geq i$ and $\textbf{y} \in \mathfrak{A}_j$ , then $\psi_{ji}(\textbf{y}) \in \mathfrak{A}_i$ and as the inverse limit $\mathfrak{A}$
of the inverse system $(A_i, \psi_{ji})_{i,j\in I}$ (since by 2. of Definition \ref{s4} $\mathfrak{A}\cong\cap (\pi_i^e)^{-1}(\mathfrak{A}_i)\neq 0$) is
nonempty closed subspace of the linearly-compact vector space $\mathfrak{g}^e\cong\varprojlim \mathfrak{g}_i^e$. Each element in $\mathfrak{A}$ defines a system of generators $\textbf{g}=(g_1, . . . , g_e)$ of $\mathfrak{g}$
with $\mathfrak{f}(\textbf{g}) = \textbf{h}$ and again $exp(\textbf{g})$ is the desired element in $G(k)$.
\end{proof}

We will need the following prounipotent analog of \cite[Proposition 2.2.]{Lub} and \cite[Proposition 2.4.]{Lub}, it has a slighter weaker form than \cite{Lub} as we do not limit to only finitely generated prounipotent groups.

\begin{lemma}\label{l9}
Let we are given $1\rightarrow R\rightarrow F\xrightarrow{\pi} G\rightarrow 1$ - a non-proper prounipotent presentation, then there is a basis $\widehat{S}^*=\widehat{X}^*\cup \widehat{Z}^*\subseteq F(k)$, such that $\widehat{Z}^*\subseteq R(k)$ and $1\rightarrow R_1\rightarrow F(X)\xrightarrow{\widehat{\pi}} G\rightarrow 1$ is the proper presentation, where $R_1$ is the image of $R$ under the natural epimorphism $F(\widehat{X}^*\cup \widehat{Z}^*)\xrightarrow{\theta} F(\widehat{X}^*)$, which sends $x\mapsto x$ if $x\in \widehat{X}^*$ and $z\mapsto  *$ if $z\in \widehat{Z}^*$.

Then $\theta$ gives rise to an exact sequence of topological $\mathcal{O}(G)^*$-modules as follows
$0\rightarrow (\mathcal{O}(G)^*)^{(\widehat{Z}^*,*)}\rightarrow \overline{R}(k)\xrightarrow{\widetilde{\theta}} \overline{R_1}(k)\rightarrow 0$.
\end{lemma}
\begin{proof} We prove in the case $char(k)=0$, the pro-$p$-case is similar. Since the presentation is not proper, the induced epimorphism of linearly-compact topological vector spaces $\widetilde{\pi}:F/[F,F](k)\rightarrow G/[G,G](k)$ has a non-zero kernel $Ker\; \widetilde{\pi}\cong R[F,F]/[F,F](k)\cong R(k)/(R(k)\cap [F,F](k))$ and therefore we get the decomposition
$$F/[F,F](k)\cong G/[G,G](k)\oplus R[F,F]/[F,F](k).$$ Let $Z$ be a basis of $(R[F,F]/[F,F](k))^{\vee}$ and let $X$ be a basis of $(G/[G,G](k))^{\vee}$. We choose the union uplifts $\widehat{Z}^*$ to $R(k)$ (along the projection $R(k)\rightarrow R(k)/(R(k)\cap [F,F](k))$) and $\widehat{X}^*$ in $F(k)$ (in the pro-$p$-case, as always, we use $F^p[F,F]$ and $G^p[G,G]$ instead $[F,F]$ and $[G,G]$ and the uplifts are continuous sections, which exist by \cite[Proposition 1]{Se4}). By construction, $\widehat{X}^*\cup \widehat{Z}^*$ has the desired properties.

Let $N=Ker\;\theta$, then, by previous discussion, we have an exact sequence $1\rightarrow N\rightarrow R\rightarrow R_1\rightarrow 1$ of free prounipotent groups. The corresponding low degree 5-exact sequence in the Lyndon-Hochschild-Serre spectral sequence of $G$-modules takes the form
$0\rightarrow H^1(R_1,k_a)\rightarrow H^1(R,k_a)\rightarrow H^1(N,k_a)^R\rightarrow 0,$ which is equivalent by duality to
$0\rightarrow \overline{R}_1(k)^{\vee}\rightarrow \overline{R}(k)^{\vee}\rightarrow (\overline{N}(k)^{\vee})^R\rightarrow 0$
and, therefore, again, by duality, we get an exact sequence of topological $\mathcal{O}(G)^*$-modules
$$0\rightarrow \overline{N}(k)_{R(k)}\rightarrow \overline{R}(k)\rightarrow \overline{R_1}(k)\rightarrow 0.$$

By Proposition \ref{p1}
$(N,F(\widehat{X}^*),d_1|_N)$, where $d_1(\widehat{Z}^*)=1$ and $Y=Z$ in our environment, is a free prounipotent pre-crossed module,
and since $Ker\; d_0=Ker\; d_1$ we see $(C=Ker\;d_0/[Ker\;d_0,Ker\;d_1], F(X), \overline{d}_1)\cong (\overline{N}(k), F(X), \overline{d}_1)$ is the free crossed module. Its universal property, using Corollary \ref{c3}, imply that
$\overline{N}(k)\cong (\mathcal{O}(F(d))^*)^{(\widehat{Z}^*,*)}$
is a free topological $\mathcal{O}(F(d))^*$-module \cite[Proposition 7]{BH1} and, therefore, taking coinvariants, as $\mathcal{O}(F)^R\cong \mathcal{O}(G)$, we obtain that $\overline{N}(k)_R\cong (\mathcal{O}(G)^*)^{(\widehat{Z}^*,*)}$.
\end{proof}

\subsection{Relation modules and aspherical presentations}\label{s3.2}
When $G$ is a prounipotent group we can obtain a precise description of injective hulls.

 \begin{proposition} \label{p18} ~\cite[~1.11]{LM}
Let $G$ be a prounipotent group, and let $V$ be a $G$-module. Then
$V^G\otimes \mathcal{O}(G)$ is an injective $G$-module containing $V$.
Each injective $G$-module containing $V$ contains a copy of $V^G\otimes \mathcal{O}(G)$.
\end{proposition}

Due to the previous statement we can define the \emph{injective hull} of a $G$-module $V$ by the formula
$\mathcal{E}_0(V) = V^G\otimes \mathcal{O}(G).$ Put $\mathcal{E}_{-1}(V)=V$, and let
$d_{-1}:\mathcal{E}_{-1}(V)\rightarrow \mathcal{E}_0(V)$ be the corresponding inclusion.

\begin{lemma} \label{l13} Let $A$ be a $G$-module over a prounipotent group $G$ then the following statements are equivalent:
\begin{description}
\item[A1]{$A^*$ is a free topological $\mathcal{O}(G)^*$-module over a pointed profinite space $(\widehat{X}^*,*)$, i.e. $A^*\cong (\mathcal{O}(G)^*)^{(\widehat{X}^*,*)}$}
\item[A2]{$A$ is an induced $G$-module over $X$, i.e. $A\cong (\oplus_X k_a)\uparrow_1^G$}
\item[A3]{$H^1(G,A)=0$}
\end{description}
\end{lemma}
\begin{proof}
By duality $A1\Leftrightarrow A2$, also $A2\Rightarrow A3$, so we must prove $A3\Rightarrow A2$.

By Proposition \ref{p18}, $A$ is a submodule of its injective hull $V=A^G\otimes \mathcal{O}(G)\cong (A^G)\uparrow_1^G$, that is, $A\xrightarrow{j} V= A^G\otimes \mathcal{O}(G)$.
Let $C=Coker\;(j)$, then we have the exact sequence
$0\rightarrow A^G\rightarrow V^G\rightarrow C^G\rightarrow H^1(G,A)\rightarrow 0$.
Since $H^1(G,A)=0$, then $C^G=0$ and hence $C=0$ and, therefore, $A\cong (A^G)\uparrow_1^G$.
\end{proof}

 \begin{proposition} \label{p19}
 ~\cite[~1.12]{LM} Let $G$ be a prounipotent group and $V$ be a $G$-module. Define the
 \textbf{minimal resolution} $\mathcal{E}_i(V)$ and $ d_i : \mathcal{E}_i(V)\rightarrow \mathcal{E}_{i+1}(V)$ inductively,
 $$\mathcal{E}_{i+1}(V) = \mathcal{E}_0(\frac{\mathcal{E}_i(V)}{d_{i-1}(V))})\quad d_i=\mathcal{E}_i(V)\rightarrow
 \frac{\mathcal{E}_i(V)}{d_{i-1}(\mathcal{E}_{i-1}(V))}\rightarrow \mathcal{E}_{i+1}.$$
 Then $\{\mathcal{E}_i( V), d_i \}$ is an injective resolution of $V$ and $H^i(G, V) = \mathcal{E}_i( V)^G.$
\end{proposition}
\begin{proof} $\mathcal{E}_i(V)=\mathcal{E}_0(Ker\; d_i)=(Ker\; d_i)^G\otimes \mathcal{O}(G)$. By construction and Proposition \ref{p18} $\{\mathcal{E}_i( V), d_i \}$ is an injective resolution, hence $\mathcal{E}_i(V)^G=(Ker\; d_i)^G$, as $\{\mathcal{E}_i( V)^G\}$ has zero differentials and, therefore, $H^i(G, V) = \mathcal{E}_i( V)^G$.
\end{proof}

\begin{proposition}\label{t3} Assume we are given a prounipotent presentation \eqref{5} of a pro-unipotent group $G$. Then
$cd(G)=2$ if and only if there is a non-canonical isomorphism of right topological $\mathcal{O}(G)^*$-modules $\overline{R}(k)\cong (\mathcal{O}(G)^*)^{(\widehat{Y}^*,*)}$.
\end{proposition}
\begin{proof} First of all, note that the proof of the isomorphism of right topological $\mathcal{O}(G)^*$-modules
$\overline{R}(k)\cong (\mathcal{O}(G)^*)^{(\widehat{Y}^*,*)}$, by continuous duality, is equivalent
to the proof of the isomorphism $Hom_{cts}(\overline{R}(k),k)\cong\mathcal{O}(G)^Y$ of right $\mathcal{O}(G)$-comodules.
\cite[Proposition 10]{Mikh2017} asserts that there is an isomorphism of $\mathcal{O}(G)$-comodules
$H^1(R,k_a)\cong Hom(\overline{R},k_a)\cong\overline{R}(k)^{\vee}$.
Thus, we must prove the isomorphism of right $\mathcal{O}(G)$-comodules
$H^1(R,k_a)\cong \mathcal{O}(G)^Y$.

Let us study the minimal injective $\mathcal{O}(G)$-resolution of the trivial
$\mathcal{O}(G)$-comodule $k_a$ (i.e. we consider a proper presentation). Since the cohomological dimension
of $F$ equals one, then the minimal $\mathcal{O}(F)$-resolution of $k_a$ takes the form
$$0\rightarrow k_a\rightarrow \mathcal{O}(F)\rightarrow \mathcal{O}(F)^X\rightarrow 0.$$ By
\cite[I, 4.12, 3.3]{Jan} we can regard this resolution as
injective $\mathcal{O}(R)$-comodule resolution. Applying the $R$-fixed points functor
and taking into account that $R$-fixed points of $\mathcal{O}(F)$ coincide with $\mathcal{O}(G)$
\cite[16.3]{Wat}, we get the exact sequence
$0\rightarrow k_a\rightarrow \mathcal{O}(G)\rightarrow \mathcal{O}(G)^X\rightarrow
H^1(R,k_a)\rightarrow 0.$
It follows from Proposition \ref{p18} that the injective
hull of $H^1(R,k_a)$ coincides with $ H^1(R,k_a)^G\otimes \mathcal{O}(G)\cong \mathcal{O}(G)^Y$, since $H^2(G,k_a)\cong H^1(R,k_a)^G\cong k^Y$.
A composition of the map
$\mathcal{O}(G)^X\rightarrow H^1(R,k_a)$
with the inclusion of $H^1(R,k_a)$ in its injective hull give rise to
minimal injective $\mathcal{O}(G)$-resolution
$0\rightarrow k_a\rightarrow \mathcal{O}(G)\rightarrow \mathcal{O}(G)^X\rightarrow\mathcal{O}(G)^Y$
of the trivial comodule $k_a$.
In particular, one has an isomorphism of $\mathcal{O}(G)$-comodules
$H^1(R,k_a)\cong Im\;(\mathcal{O}(G)^X\rightarrow \mathcal{O}(G)^Y).$
From Proposition \ref{p19} it follows that $\{\mathcal{E}_i( V)^G\}$ has zero differentials and, therefore, since nontrivial modules over prounipotent groups have nontrivial fixed points, we see that conditions $H^1(R,k_a)\cong\mathcal{O}(G)^Y$  and $cd(G)= 2$ are equivalent.

If the presentation is not proper, then, by Lemma \ref{l9} and above considerations, we get a short exact sequence of topological $\mathcal{O}(G)^*$-modules
$$0\rightarrow (\mathcal{O}(G)^*)^{(\widehat{Z}^*,*)}\rightarrow \overline{R}(k)\xrightarrow{\widetilde{\theta}} (\mathcal{O}(G)^*)^{(\widehat{Y\setminus Z}^*,*)}\rightarrow 0,$$
where we use a slightly non-precise notations, identifying $Y$ with a special new basis in $R/[R,F](k)$, but this is not important (see also Proposition \ref{t2} below). Indeed, the cardinality of $Y$ is the sum of cardinalities of the basis of $(R[F,F]/[F,F](k))^{\vee}$ and the cardinality of a basis of the complement to $(R[F,F]/[F,F](k))^{\vee}\hookrightarrow (R/[R,F](k))^{\vee}$.
This sequence splits along the embedding $(\widehat{Y\setminus Z}^*,*)\hookrightarrow (\widehat{Y}^*,*)$ and therefore
$\overline{R}(k)\cong (\mathcal{O}(G)^*)^{(\widehat{Y}^*,*)}$. We only need to check, that $Y\setminus Z$ has the same cardinality as a basis of $(R_1/[R_1,F(X)](k))^{\vee}$. By construction $1\rightarrow R_1\rightarrow F(X)\xrightarrow{\widehat{\pi}} G\rightarrow 1$ is the proper presentation and therefore $(R_1/[R_1,F(X)](k))^{\vee}\cong H^1(R,k_a)^{F(X)}\cong H^2(G,k_a).$
The 5-exact sequence in the Lyndon-Hochschild-Serre spectral sequence of a presentation $1\rightarrow R\rightarrow F\xrightarrow{\pi} G\rightarrow 1$ takes the form
$$0\rightarrow H^1(G,k_a)\xrightarrow{Inf} H^1(F,k_a)\rightarrow H^1(R,k_a)^F\rightarrow H^2(G,k_a)\rightarrow 0.$$
By construction $Coker\;(Inf)=k^Z$, and, therefore, we have $0\rightarrow k^Z\rightarrow (R/[R,F](k))^{\vee}\rightarrow (R_1/[R_1,F(X)](k))^{\vee}\rightarrow 0$ - the exact sequence of vector spaces,
which imply that $k^{Y\setminus Z}$ ($Y\setminus Z$ - in the sense of cardinality of sets) is the basis of $(R_1/[R_1,F(X)](k))^{\vee}$ and hence (by Lemma \ref{l13}) $R_1/[R_1,F(X)](k)\cong k^{(\widehat{Y\setminus Z}^*,*)}$. But since $\overline{R_1}(k)$ is free we have $\overline{R_1}(k)\cong (\mathcal{O}(G)^*)^{(\widehat{Y\setminus Z}^*,*)}$.
\end{proof}
Proposition \ref{t2} below says that a presentation \eqref{5} is aspherical if and only if its relation module is free. It is clear from this result that $(R/[R,F](k))^{\vee}\cong k^Y$ if $u_2(X|R)=0$.

\begin{proposition}\label{t2} Assume we are given a simplicial prounipotent presentation \eqref{5} of a prounipotent group $G$. Then
$u_2(X|R)=0$ if and only if there is an isomorphism of topological $\mathcal{O}(G)^*$-modules $\overline{R}(k)\cong (O(G)^*)^{(\widehat{Y}^*,*)},$ i.e. $\overline{R}(k)$ is a free topological $\mathcal{O}(G)^*$-module over the pointed profinite space $(\widehat{Y}^*,*)$.
\end{proposition}
\begin{proof}
Since $Im(\overline{d}_1)(k)=R(k)$ is the subgroup of $F(k)$ it is also free \cite[Corollary 2.10]{LM} and, by Lemma \ref{l1234}, $\overline{d}_1(k)$ has a section, say $s$, and therefore $C(k)\cong u_2(X|R)\times sR(k)$. Indeed, $u_2(X|R)$ is central in $C$, as for $p\in C(k), a\in u_2(X|R)$ the following equalities hold $p^{-1}a^{-1}pa=p^{-1}p^{\partial a} a=1$. But since $u_2(X|R)$ is abelian it follows that $[C(k),C(k)]=[sR(k),sR(k)]$ and therefore $[C(k),C(k)]\cap u_2(X|R)=1$ and, therefore, $u_2(X|R)=Ker\;({\overline{C}(k)\rightarrow \overline{R}(k)})$.
Now it follows from Proposition \ref{p1} that $\overline{C}(k)\cong (\mathcal{O}(G)^*)^{(\widehat{Y}^*,*)}$
(see also \cite[Corollary 3.17]{Mikh2016}) and the result is clear if the presentation is proper.

Let us consider a non-proper presentation. As $\overline{R}(k)$ is a free topological $\mathcal{O}(G)^*$-module, then its dual $(\overline{R}(k))^{\vee}$ is the injective $G$-module and, therefore, the image of the dual homomorphism $\overline{d}_1^{\vee}:(\overline{R}(k))^{\vee}\rightarrow (\overline{C}(k))^{\vee}$ is a direct summand, i.e. there is a $G$-module decomposition $(\overline{C}(k))^{\vee}\cong (\overline{R}(k))^{\vee}\oplus B$. By Remark \ref{r9}, the induced homomorphism on fixed points $(\overline{d}_1^{\vee})^G:((\overline{R}(k))^{\vee})^G\rightarrow ((\overline{C}(k))^{\vee})^G$ is the isomorphism, hence $B^G=0$. By Lemma \ref{l2}, $B=0$ and therefore $\overline{d}_1^{\vee}$ is the isomorphism, i.e. $u_2(X|R)=0$.
\end{proof}

\subsection{Tsvetkov's criterion}\label{s1.4}
Consider the diagram \eqref{8} of prounipotent groups arising from a prounipotent presentation \eqref{5} of a prounipotent group $G$ with $cd(G)=2$ and its subpresentation \eqref{6} of a prounipotent group $G_0$. In the case of positive characteristics we always assume that the base field $k=\mathbb{F}_p$ and $G$ is an affine group scheme, defined as in Example \ref{e6}.
\begin{equation}\label{8}
\xymatrix{1 \ar[r] & H \ar[r] & G_0 \ar[r]^{\pi} & G \ar[r] & 1 \\
  && F \ar[u]^{\rho} \ar[ur]_{\pi\circ\rho} \\
  & R \ar[ur] \ar[uu] & R_0 \ar[u] }
\end{equation}
The inclusion of normal subgroups $R_0\subset R$ of $F=F(X)$ induces a homomorphism of relation modules
$$\phi:\overline{R}_0(k)\rightarrow \overline{R}(k),$$
that is, a homomorphism of topological $\mathcal{O}(G_0)^*$-modules with the module structure of Corollary \ref{c1}.
Since $R_0$ is normal in $F$, $Im\;\phi$ is a topological $\mathcal{O}(G)^*$-submodule of $\overline{R}(k)$.
We remind that in the pro-$p$-case we always assume $\overline{R}=R/R^p[R,R]$ is a topological $\mathbb{F}_pG$-module (i.e. $\mathcal{O}(G)^*= \mathbb{F}_pG$).

Short exact sequence of groups
$1\rightarrow R_0\rightarrow R\rightarrow H\rightarrow 1$
gives rise to a low degree 5-exact sequence in the Lyndon-Hochschild-Serre spectral sequence
$$0\rightarrow H^1(H,k_a)\xrightarrow{Inf} H^1(R,k_a)\xrightarrow{Res} H^1(R_0,k_a)^{R/R_0}\xrightarrow{Tg} H^2(H,k_a)\rightarrow 0,$$
where $H^2(R,k_a)=0$ since $R$ is free.

The image of the restriction map $Res:H^1(R,k_a)\rightarrow H^1(R_0,k_a)$ in the 5-exact sequence coincide with $Im\;\phi^{\vee}$ and in particular, since $Im(Res)\subseteq H^1(R_0,k_a)^{R/R_0}$, $Im(Res)$ is a $G$-module.
Consider two exact sequences of $G$-modules
\begin{equation}\label{eq11}
0\rightarrow H^1(H,k_a)\xrightarrow{Inf} H^1(R,k_a)\rightarrow Im\;\phi^{\vee}\rightarrow 0
\end{equation}
and
\begin{equation}\label{eq12}
0\rightarrow Im\;\phi^{\vee}\rightarrow H^1(R_0,k_a)^{R/R_0}\xrightarrow{Tg} H^2(H,k_a)\rightarrow 0
\end{equation}

\label{s1}
\begin{proposition}\label{t1}
Suppose we are given a diagram \eqref{8} of prounipotent groups, then the following conditions are equivalent:
\begin{description}
\item[(i)]{$cd\;G_0\leq2$}
\item[(ii)]{$cd\;H\leq1$}
\end{description}
and
\begin{description}
\item[(I)]{$Im\;\phi$ is the free topological $\mathcal{O}(G)^*$-module}
\item[(II)]{$H^0(G,Im\;\phi^{\vee})\cong H^0(G_0,H^1(R_0,k_a))$ - the map arising by applying the fixed points functor to \eqref{eq12} is the isomorphism }
\end{description}
\end{proposition}

\begin{proof}
First, we show that the conditions $(I)$ and $(II)$ are equivalent to:
\begin{description}
\item[(*)]{$H^1(G,Im\;\phi^{\vee})=0;$}
\item[(**)]{$H^0(G,Im\;\phi^{\vee})\cong H^0(G,H^1(R_0,k_a)^{R/R_0})$ - the map arising by applying the fixed points functor to \eqref{eq12} is the isomorphism }
\end{description}
For any $G$-module $A$ and a normal subgroup $H\lhd G$ one has $(A^H)^{G/H}=A^G$, taking $R/R_0\lhd G_0$, we get $H^0(G_0,H^1(R_0,k_a))\cong H^0(G,H^1(R_0,k_a)^{R/R_0})$ and therefore $(**)$ is equivalent to $(II)$.
By Lemma \ref{l13} (I) is equivalent to $H^1(G,(Im\;\phi)^{\vee})=0$ and, as $Im\;\phi^{\vee}\cong (Im\;\phi)^{\vee}$, (I) is equivalent to (*).

Now we prove that (*) and (**) are equivalent to:
\begin{description}
\item[(a)]{$H^1(G,H^1(H,k_a))\rightarrow H^1(G,H^1(R,k_a))\mbox{ is the epimorphism};$}
\item[(b)]{$H^2(G,H^1(H,k_a))\rightarrow H^2(G,H^1(R,k_a))\mbox{ is the monomorphism};$}
\item[(c)]{$H^0(G,H^2(H,k_a))=0$.}
\end{description}

This follows from the long exact sequences of cohomology of \eqref{eq11} and \eqref{eq12}:

$$0\rightarrow H^0(G,H^1(H,k_a))\rightarrow H^0(G,H^1(R,k_a))\rightarrow H^0(G,Im\;\phi^{\vee})$$
$$\rightarrow H^1(G,H^1(H,k_a))\rightarrow H^1(G,H^1(R,k_a))\rightarrow \mathbf{H^1(G,Im\;\phi^{\vee})}$$
$$\rightarrow H^2(G,H^1(H,k_a))\rightarrow H^2(G,H^1(R,k_a))\rightarrow H^2(G,Im\;\phi^{\vee})\rightarrow \ldots$$

and

$$0\rightarrow \mathbf{H^0(G,Im\;\phi^{\vee})}\rightarrow \mathbf{H^0(G,H^1(R_0,k_a)^{R/R_0})}\rightarrow H^0(G,H^2(H,k_a))$$
$$\rightarrow \mathbf{H^1(G,Im\;\phi^{\vee})}\rightarrow H^1(G,H^1(R_0,k_a)^{R/R_0})\rightarrow H^1(G,H^2(H,k_a))^{R/R_0}\rightarrow \ldots$$
The next step is to show that conditions  (a), (b), (c) are equivalent to:
\begin{description}
\item[(d)]{$d_2^{1,1}: H^1(G,H^1(R,k_a))\rightarrow H^{3}(G,k_a)\mbox{ is the epimorphism};$}
\item[(e)]{$d_2^{2,1}: H^2(G,H^1(R,k_a))\rightarrow H^{4}(G,k_a)\mbox{ is the monomorphism};$}
\item[(f)]{$cd\;(H)\leq 1.$}
\end{description}

By Lemma \ref{l2}, (c) is equivalent to $H^2(H,k_a)=0$ and hence to $cd\;H\leq1$ (i.e. $H$ is free), therefore (c) is equivalent to (f).
We arrange (a) and (b) into the commutative diagram
$$\xymatrix{
  H^p(G,H^1(H,k_a)) \ar[rr]^{Inf^*} \ar[dr]_{\widetilde{d}_2^{p,1}}
                &  &    H^p(G,H^1(R,k_a)) \ar[dl]^{d_2^{p,1}}    \\
                & H^{p+2}(G,k_a)               },$$
where $Inf^*$ is induced by the inflation map $Inf:H^1(H,k_a)\rightarrow H^1(R,k_a)$, $Inf^*$ coincides with (a), when $p=1$, and with (b), when $p=2$. The maps $\widetilde{d}_2^{p,1}$ and $d_2^{p,1}$ are the differentials in the spectral sequences $\widetilde{E}_2^{p,q}$, $E_2^{p,q}$ of group extensions $1\rightarrow H\rightarrow G_0\rightarrow G\rightarrow 1$ and $1\rightarrow R\rightarrow F\rightarrow G\rightarrow 1,$ respectively.
Lets take $G_0$-module resolution \eqref{eq5} of the trivial module $k_a$
and $F$-module resolution \eqref{eq5} of the trivial module $k_a$, these resolution are also injective $H$ and $R$-module resolutions of the trivial module $k_a$.
Then $Inf$ is induced on the cohomology by the inclusion of the complexes \cite[6.10]{Jan}
$$(k_a\otimes \mathcal{O}(G_0)^{\otimes(q+1)})^H\rightarrow (k_a\otimes \mathcal{O}(F)^{\otimes(q+1)})^R.$$
Let $\gamma_0: \widetilde{E}_0^{p,q}\rightarrow E_0^{p,q}$ be the induced inclusion of the double complexes
$$(k_a\otimes \mathcal{O}(G_0)^{\otimes(q+1)})^H\otimes \mathcal{O}(G)^{\otimes p}\rightarrow (k_a\otimes \mathcal{O}(F)^{\otimes(q+2)})^R\otimes \mathcal{O}(G)^{\otimes p},$$
then $\gamma_0$ commute with $d_0^{p,q}$, $\widetilde{d}_0^{p,q}$, $d_1^{p,q}$, $\widetilde{d}_1^{p,q}$ and, therefore, by \cite[VIII,9]{HSJ}, $\gamma_0$ induces a homomorphism of spectral sequences $\gamma_n:\widetilde{E}_n^{p,q}\rightarrow E_n^{p,q}$ and in particular $\gamma_2:H^p(G,H^1(H,k_a))\rightarrow H^p(G,H^1(R,k_a))$. By construction, $\gamma_2$ coincides with $Inf^*$ (an explicit mapping is given in Remark \ref{r3.1} and therefore $\widetilde{d}^{p,1}_2=d_2^{p,1}\circ Inf^*$.

The Lyndon-Hochschild-Serre spectral sequence of the group extension
$$1\rightarrow R\rightarrow F\rightarrow G\rightarrow 1$$
degenerates (since $R,F$ are free), hence $d_2^{p,1}$ is an isomorphism and therefore conditions (a), (b) are equivalent to (d), (e).

Let us also consider the Lyndon-Hochschild-Serre spectral sequence $E_2^{p,q}=H^p(G,H^q(H,k_a))$ of the group extension
$1\rightarrow H\rightarrow G_0\rightarrow G\rightarrow 1$
which also degenerates for $q>1$, since $H$ is free and, therefore, $E_3^{p,q}\cong E_{\infty}^{p,q}.$

A filtration of $H^{p+q}(G_0,k_a)$ provided by $E_{\infty}^{p,q}$ is given by the formulae as follows
$$E_{\infty}^{p,q}\cong gr_pH^{p+q}(G,k_a)=F^pH^{p+q}(G,k_a)/F^{p+1}H^{p+q}(G,k_a)$$
and therefore (d),(e),(f) are equivalent to $H^3(G_0,k_a)=0,$ $cd\;H\leq 1,$ i.e. conditions (i), (ii) of Proposition \ref{t1}.
\end{proof}

\subsection{Proof of Theorem \ref{t4}}\label{s3.5}
We use notations of sections \ref{s2.2}, \ref{s1.4}, i.e. $R_0(k)=\langle \tau(\widehat{Y}_1^*)\rangle_{F(k)}$, where $Y_1\subset Y$.
\begin{proof} The inclusion of normal subgroups $R_0\subset R$ induces an abelianization homomorphism $\phi:\overline{R_0}(k)\rightarrow \overline{R}(k)$.
First of all, note that
$ Im\;\phi\cong (\mathcal{O}(G)^*)^{(\widehat{Y_1}^*,*)}$
is a free topological module on the profinite space $(\widehat{Y_1}^*,*)$. By Proposition \ref{t2}, $\overline{R}(k)\cong (\mathcal{O}(G)^*)^{(\widehat{Y}^*,*)}$ and $Im\; \phi$ is a closed submodule of $\overline{R}(k)$ generated as a topological $\mathcal{O}(G)^*$-submodule by $(\widehat{Y_1}^*,*)$.
Now we want to show, that \eqref{eq12} induces the isomorphism $H^0(G,Im\;\phi^{\vee})\cong H^0(G_0,H^1(R_0,k_a)).$

The last isomorphism is equivalent to the isomorphism $(Im\;\phi^{\vee})^{G}\cong H^1(R_0,k_a)^{G_0}$
of $k$-spaces arising from the embedding of \eqref{eq12}
 and by continuous duality to the isomorphism
$((Im\;\phi^{\vee})^G)^*\cong ((H^1(R_0,k_a))^{G_0})^*$.
By Lemma \ref{l4} that is equivalent to the isomorphism $(Im\;\phi^{\vee})^*_G\cong (H^1(R_0,k_a)^*)_{G_0}.$ But $Im\;\phi^{\vee}\cong (Im\;\phi)^{\vee}$ and so we must check that taking $G$-coinvariants of $\phi$ imply the isomorphism
$(Im\;\phi)_G\cong R_0/[R_0,F]$.
Indeed,
$(Im\;\phi)_G\cong ((\mathcal{O}(G)^*)^{(\widehat{Y_1}^*,*))}_G\cong k^{(\widehat{Y_1}^*,*)}$ and $R_0/[R_0,F]\cong k^{(\widehat{Y_1}^*,*)}$, as $R_0$ is taken from \eqref{6}.

We are now ready to apply Proposition \ref{t1} (which shows that $cd\; G_0\leq 2$) and Propositions \ref{t3},  \ref{t2} in order to show that $(X|R_0)$ is aspherical.
\end{proof}

\section{An application: discrete group presentations}\label{s6}

In the recent paper Berrick and Hillman have obtained  the following beautiful consequence \cite[Corollary 4.8]{BH}.

\begin{proposition} The Whitehead Conjecture is equivalent to the conjecture that the fundamental group of a finite subcomplex of a contracnible 2-complex has rational cohomological dimension at most 2.
\end{proposition}

This result eliminate the crucial role of finite subcomplexes of contractible 2-complexes in the study of asphericity.

From now and on, we use the standard notation of the section \ref{s2.2} with the only difference that for brevity we will identify in the notation the normal subgroup and the relations generating it.
As always, we replace arbitrary contractible 2-complex by a complex of a presentation $(X|Y)$ of the trivial group. 
As in Section \ref{r1.1} consider its combinatorial model $S_1K(X|Y)$.
 In what follows we shall prove that the Bousfield-Kan $R$-completion $R_{\infty}S_1K(X|Y)$ is an aspherical space.

We remind (see Section \ref{s0}) that $R_{\infty}S_1K(X|Y)$  may be obtained as the classifying space $\overline{W}((GS_1K(X|Y))^{\wedge}_R)$ of the dimension-wise $R$-nilpotent completion of the Kan loop-group $G$ of $S_1K(X|Y)$. 
Recal, by Section \ref{r1.1} the geometric realization of $GS_1K(X|Y)$ is homotopy equivalent to $\Omega K(X|Y)$.
Since $(X_1|Y_1)$ is a finite presentation it follows that the dimension-wise $R$-nilpotent completion $(GS_1K(X_1|Y_1))^{\wedge}_R$ is homotopy equivalent to the corresponding pro-unipotent simplicial group presentation $(X_1|Y_1)_u$ \cite[(3.3)]{Mikh2016}, see Section \ref{s2.1}. 

\begin{corollary} \label{c4} Let $(X_1|Y_1)$ be a finite subpresentation of a contractible presentation $(X|Y)$ and $R=k$ is a field of zero characteristics or $R=\mathbb{F}_p$, then $R_{\infty}S_1K(X_1|Y_1)$ is aspherical.
\end{corollary}
\begin{proof} Since $(X|Y)$ is contractible it is also aspherical, in our case $R=\Phi(X)$ and $R/[R,R]$ is forced to be isomorphic to $\mathbb{Z}G^Y$ (see Section \ref{r1.2}), therefore we obtain a non-canonical isomorphism
$ \mathbb{Z}^X\cong \Phi(X)/[\Phi(X),\Phi(X)]\cong \mathbb{Z}^Y.$
Now we consider the corresponding prounipotent presentation $(\widehat{X}^*|Y)$ of the trivial prounipotent group, so $R_u=\langle Y\rangle_F=F(\widehat{X}^*,*)$ is the Zariski closure of defining relations $r_y\in\Phi(X), y\in Y$ in the the free prounipotent group $F=F(X)(k)$.

Since the images of $r_y, y\in Y$ generate $F/[F,F]$ and are linearly independent, it follows from free prounipotent group construction (see Sections \ref{s1.2} and \ref{s2}, Proposition \ref{p4}) that they are generators of $F$, i.e $F(X)(k)= F(Y)(k) $. 
By assumption, $X_1$ and $Y_1$ are finite, so we may choose elements of $Y_1$ as generators of $F(X_1)(k)$ and therefore $F(X_1)/(Y_1)_u\cong F(\widetilde{Y}_1)$ for some $\widetilde{Y}_1\subseteq X_1$, i.e. $(X_1|Y_1)_u$ is a presentation of a free prounipotent group, hence $cd\; F(X_1)/(Y_1)_u\leq 1$. 
By Propositions \ref{t3}, \ref{t2} (or by Theorem \ref{t4}), $(X_1|Y_1)_u$ is an aspherical simplicial prounipotent presentation and therefore $\overline{W}(X_1|Y_1)_u$ is an aspherical simplicial set in the common sense.
But $X_1$ and $Y_1$ are finite and, as mentioned above, $R_{\infty}S_1K(X_1|Y_1)\simeq\overline{W}(X_1|Y_1)_u$ is a homotopy equivalence and the result follows.
\end{proof}
The arithmetic square argument imply the following consequence.
\begin{corollary}\label{c5} Let $(X_1|Y_1)$ be a finite subpresentation of a contractible presentation $(X|Y)$ then $\mathbb{Z}_{\infty}S_1K(X_1|Y_1)$ is aspherical.
\end{corollary}

\begin{remark}\label{rem5}
In \cite{BDH} it was shown (see also \cite{D}) that any homotopy type $ X $ can be represented by a pair of discrete groups $ (G_X, P_X) $, where $ P_X $ is a normal perfect subgroup in $ G_X $ (a group $P$ is said to be perfect if $P=[P,P]$). 
In more detail, $ X $ can be recovered in the homotopy category from the Quillen plus-construction $ (K (G_X, 1), P_X)^+ $. 
Thus, the category of connected $ CW $-complexes and pointed homotopy classes of maps is equivalent to the category of fractions of the category of pairs $ (G, P) $, where $ G $ is a group and $ P $ is a perfect normal subgroup. 

One can think of Bousfield-Kan $ R $-completion as a homotopically reasoned functorial way of replacing the usual homotopy type with a homotopy type that escapes perfect subgroups.
Corollaries \ref{c4} and \ref{c5} show that ``$ R$-extraction'' of perfect subgroups turns finite subrepresentations of contractible presentations into aspherical ones and thus pushing Whitehead's conjecture now to the plane of perfect subgroups. 
It is noteworthy that in the category of prounipotent groups, where there are no nontrivial perfect subgroups, asphericity is completely determined, as it should be, with the help of $K (G, 1)$  (namely Propositions \ref{t3},\ref{t2}) by the cohomological dimension.
\end{remark}

The author is  sincerely gratefull to the referee for insisting that I not omit an application of the main result. 
The author would also like to thank the referee for the useful reference \cite[Proposition 2.6.]{CCH}, which suggests another way to get some of the paper's results in the pro-$p$-case.

%
%



\bibliography{Mikhovich_Whitehead_Arxiv1}{}
\bibliographystyle{plain}

\end{document}